\documentclass[11pt,reqno]{amsart}
\usepackage[margin=4cm]{geometry}
\usepackage{amsmath,amsfonts,amsthm,amssymb,amscd,verbatim,graphicx,color,multirow,booktabs,caption,bm,chngcntr,adjustbox,longtable,mathtools,microtype,xfrac,thmtools}
\usepackage[pagebackref]{hyperref} 
\usepackage[dvipsnames,table]{xcolor}
\definecolor{BlueSlate}{HTML}{2F4D73}
\hypersetup{colorlinks=true, linkcolor=Maroon, citecolor=OliveGreen, urlcolor=BlueSlate}
\usepackage{tikz,tikz-cd}
\usepackage[abbrev,alphabetic]{amsrefs}  

\usepackage[capitalise,noabbrev]{cleveref}
\usepackage[shortlabels]{enumitem}
\usepackage{todonotes}

\newtheorem{theorem}{Theorem}[section]
\newtheorem{lemma}[theorem]{Lemma}
\newtheorem{corollary}[theorem]{Corollary}
\newtheorem{proposition}[theorem]{Proposition}
\numberwithin{equation}{section}

\theoremstyle{definition}
\newtheorem{definition}[theorem]{Definition}

\newtheorem{remark}[theorem]{Remark}

\newcommand{\A}{\mathrm A}
\newcommand{\C}{\mathrm C}

\newcommand{\V}{V}

\newcommand{\vC}{\vec{\C}}
\newcommand{\vGa}{\vec{\Gamma}}

\newcommand{\ZZ}{\mathbb{Z}}

\newcommand{\Spl}{\mathrm s}

\newcommand{\Fix}{\mathrm{Fix}\,}

\newcommand{\Cay}{\mathrm{Cay}}

\newcommand{\Aut}{\mathrm{Aut}}

\newcommand{\Alt}{\mathrm{Alt}}
\newcommand{\Sym}{\mathrm{Sym}}
\newcommand{\SL}{\mathrm{SL}}

\newcommand{\PSL}{\mathrm{PSL}}
\newcommand{\PGL}{\mathrm{PGL}}

\def\Z#1{{\bf Z}{{(#1)}}}

\begin{document}

\title[Base size of cubic graphs]{The base size of vertex-transitive cubic graphs}

\author[Marco Barbieri]{Marco Barbieri}
\address{Fakulteta za Matematiko in Fiziko, Univerza v Ljubljani, Jadranska ulica 21, 1000 Ljubljana, Slovenia} 
\email{marco.barbieri@fmf.uni-lj.si}

\author[Luca Sabatini]{Luca Sabatini}
\address{Mathematics Institute, University of Warwick, United Kingdom} 
\email{luca.sabatini@warwick.ac.uk, sabatini.math@gmail.com}

\author[Pablo Spiga]{Pablo Spiga}
\address{Dipartimento di Matematica e Applicazioni, University of Milano-Bicocca, Via Cozzi 55, 20125 Milano, Italy} 
\email{pablo.spiga@unimib.it}

\begin{abstract}
We prove that
if $\Gamma$ is a finite connected vertex-transitive cubic graph,
then either $|\V\Gamma| \le 90$, or $\Gamma$ is a split Praeger--Xu graph,
or there exist two vertices $\alpha$ and $\beta$ such that the identity
is the only automorphism of $\Gamma$ fixing both $\alpha$ and $\beta$.
\end{abstract}


\thanks{MB is supported by the Slovenian Research Agency grant J1-50001 and research programme P1-0222.
LS is supported by the Royal Society.
All authors are also members of the GNSAGA INdAM research group and gratefully acknowledge its support.}

\maketitle

\setcounter{tocdepth}{1}
\tableofcontents

\section{Introduction} \label{sec:intro}

\subsection{Main result}

For a permutation group on a finite domain $\Omega$,
a {\em base} is a subset of $\Omega$ whose pointwise stabiliser is trivial,
and the {\em base size} is the minimum cardinality among the bases.
For instance, the base sizes of the symmetric group and of the alternating group of degree $n$
are $n-1$ and $n-2$ respectively,
while the base size of a general linear group coincides with the dimension of the underlying vector space.
Bounding the base size is a fundamental problem in permutation group theory.
On a theoretical level, an element of a permutation group is uniquely determined by its action on a base,
and on a computational level, this can be leveraged to store and perform algorithms efficiently \cite{SeressBook}.

It is natural and interesting to study the base size
when the permutation group arises as the automorphism group of some object,
and this is the approach we pursue in this paper.
In particular, we will focus on finite graphs.
Just as transitive permutation groups form the natural setting for studying base sizes,
connected vertex-transitive graphs provide the corresponding setting for graphs.
The first genuinely nontrivial case is that of cubic graphs, that is, regular graphs of valency $3$.

Our main result reveals a surprising trichotomy:
except for some small cases that can be considered low-level noise,
and a certain natural infinite class of graphs,
the base size of the automorphism group of a connected vertex-transitive cubic graph is at most $2$. 

\begin{theorem} \label{thrm:main}
Let $\Gamma$ be a finite connected vertex-transitive cubic graph with base size greater than $2$.
Then one of the following holds:
\begin{description}
\item[(i)] $|\V\Gamma| \le 90$ and $\Gamma$ is one of the distance-transitive graphs in \cref{tab:main};
\item[(ii)] $\Gamma$ is a split Praeger--Xu graph $\Spl\C(r,s)$ with $r \geq 3$ and $1 \le s < r/2$.
\end{description} 
\end{theorem}

\begin{table}[ht] 
\begin{center}
\begin{tabular}{||c||c|c|c|c|c||}
\hline\hline
&$|\V\Gamma|$ &$\Aut(\Gamma)$ &$\Aut(\Gamma)_\alpha$ & Base size & Comments \\
 \hline\hline
1&$4$  & $\Sym(4)$ &$\Sym(3)$ &  3 & complete graph ${\bf K}_4$  \\
2&$6$  & $\Sym(3)\wr_2 \Sym(2)$ &$\Sym(3)\times\Sym(2)$ &  4 & complete bipartite ${\bf K}_{3,3}$  \\  
3&$8$  & $\Sym(4)\times\Sym(2)$ &$\Sym(3)$ &  3 & cube graph\\  
4&$10$  & $\Sym(5)$ &$\Sym(3)\times\Sym(2)$ &  3 & Petersen graph\\
5&  $14$  & $\PGL_2(7)$ &$\Sym(4)$ &  3 & Heawood graph\\
 6&   $18$ & $3_+^3:D_4$ &$\Sym(3)\times\Sym(2)$ &  3 & Pappus graph\\
 7&  $20$  & $\Sym(5)\times\Sym(2)$ &$\Sym(3)\times\Sym(2)$ &  3 & Desargues graph\\    
 8&  $30$ & $\mathrm{P}\Gamma\mathrm{L}_2(9)$ &$\Sym(4)\times\Sym(2)$ &  3 & Tutte--Coxeter graph\\ 
 9&    $90$ &  $3.\mathrm{P}\Gamma\mathrm{L}_2(9)$ &$\Sym(4)\times\Sym(2)$ &  3 & Foster graph\\ 
\hline\hline
\end{tabular}
\medskip
\caption{Exceptional graphs in \cref{thrm:main}} 
\label{tab:main}
\end{center}
\end{table}

We refer the reader to \cref{sec:PX} for the definition of the Praeger--Xu graphs and their splits.
The exceptions in \cref{thrm:main} can be understood as the most symmetric among the connected cubic graphs.
Hence, it is not unexpected that several notable graphs appear in \cref{tab:main}.

\subsection{Consequences and open directions}

A striking feature of vertex-transitive cubic graphs is the gap between
the exponential growth of the automorphism groups of split Praeger--Xu graphs and all other cases.
The best possible bound for the order of these groups is of type $O(|V\Gamma|^2/\log|V\Gamma|)$
and has been obtained in~\cite{PotocnikSpigaVerret2015}.
Although it gives a weaker inequality, \cref{thrm:main} provides a much more satisfactory explanation of this phenomenon.

\begin{corollary}
Let $\Gamma$ be a finite connected vertex-transitive cubic graph.
Then either $\Gamma$ is a split Praeger--Xu graph, or $|\Aut(\Gamma)| \le 2|\V\Gamma|^2$.
\end{corollary}
\begin{proof}
If $\Aut(\Gamma)$ has base size at most $2$,
then $|\Aut(\Gamma)| \leq |\V\Gamma|^2$ \cite[Exercise~3.3.2]{DixonMortimer}.
The proof follows by a direct inspection of the graphs in \cref{tab:main}.
\end{proof}

Recently there has been interest in stabilisers which are not necessarily trivial \cite{Babai2022,Gluck2025,Sabatini2026}.
The following statement in this spirit turns \cref{thrm:main} into a uniform result.

\begin{corollary}
Let $\Gamma$ be a finite connected vertex-transitive cubic graph.
Then there exist two vertices of $\Gamma$ whose pointwise stabiliser is abelian.
\end{corollary}
\begin{proof}
The vertex-stabiliser in a split Praeger--Xu graph has an elementary abelian subgroup
of index at most $2$, and it is easy to see that the claim holds in this case.
The proof follows by a direct inspection of the graphs in \cref{tab:main}.
\end{proof}

The {\em distinguishing number} of a permutation group is the minimum number of colours
needed to colour the domain so that only the identity preserves the colouring.
 The distinguishing number of a connected vertex-transitive cubic graph is at most $2$ unless $\Gamma$
    is either ${\bf K}_4$, ${\bf K}_{3,3}$, the cube graph or the Petersen graph \cite[Theorem~1.1]{Imrich2019}.
The {\em distinguishing cost} is the minimum size of a colour class in a symmetry-breaking colouring. 
We will prove that, among the $2$-distinguishable cubic vertex-transitive graphs,
the only ones with large distinguishing cost are the split Praeger--Xu graphs.
This answers \cite[Question~9.12]{Imrich2022}.

\begin{corollary} \label{cor:distinguishinCost}
    Let $\Gamma$ be a finite connected vertex-transitive cubic graph with distinguishing number $2$.
    Then the distinguishing cost is at most $3$ unless $\Gamma$ is a split Praeger--Xu graph.
\end{corollary}

The proof of \cref{cor:distinguishinCost} is slightly more involved, and we delay it to \cref{sec:corProof}.

It is natural to ask whether a result in the spirit of \cref{thrm:main} should hold for graphs of higher valency.
A major obstacle in this direction is the absence of a clear analogue of the exceptional family appearing in the cubic case:
in fact, it remains a longstanding open problem in the theory of permutation groups acting on graphs to identify a natural generalization of the Praeger--Xu graphs and their splits for valency $5$ and beyond.
Nevertheless, we believe that a similar phenomenon should hold at least for $4$-valent vertex- and edge-transitive graphs,
where we expect the exceptional family to be given by the (nonsplit) Praeger--Xu graphs.
In the course of our proof we identify two specific points where our argument fails to extend to this setting
(see Remarks~\ref{rem:prob1} and \ref{rem:prob2}).

It would be highly interesting to have a conceptual reason
for why large cubic graphs have base size either at most $2$ or unbounded.
The structure and depth of the proof of \cref{thrm:main},
which depends on the Classification of the Finite Simple Groups,
reflect this lack of a simple explanation.

\subsection{Structure of the proof}\label{sec:readersGuide}

The proof of \cref{thrm:main} combines tools from graph theory and permutation group theory,
together with structural results on $p$-groups and almost simple groups.
For the convenience of the reader, we provide a brief roadmap of the argument. 

Let $\Gamma$ be a connected vertex-transitive cubic graph,
and let $G \leqslant \Aut(\Gamma)$ act transitively on $\V\Gamma$ with base size greater than $2$.
If $G$ is transitive on the arcs of $\Gamma$, then a simple observation shows that $\Gamma$ has girth at most $10$.
By work of Conder, Lorimer, and Morton \cite{ConderLorimer1989,Morton1991},
only finitely many graphs satisfy this condition, so we can conclude by an explicit computation (\cref{prop:arctran}).

We may therefore assume that $G$ is not arc-transitive, and the stabiliser of a vertex $G_\alpha$ is a $2$-group.
We proceed by induction on $|\V\Gamma|$,
using the normal quotient technique for graphs introduced in \cite{Praeger1993}.
A crucial ingredient is a theorem of Djokovi\'c \cite{Djokovic1980} on locally-$D_4$ group amalgams.
In particular, we observe that $G_\alpha$ has nilpotency class at most $2$ and exponent at most $4$.

At this stage we consider the structure of the socle.
If $G$ contains an abelian minimal normal subgroup $N$,
the argument reduces to a case-by-case analysis according to the valency of the quotient $\Gamma/N$.
In this situation Praeger--Xu graphs, their splits, and their covers play a central role,
and we devote \cref{sec:PX} to their study.
The analysis in \cref{sec:PX,sec:abelminnor} makes extensive use of tools from the theory of $p$-groups,
and the final result is given in \cref{prop:minabel}.

Suppose instead that the soluble radical of $G$ is trivial.
If there is more than one minimal normal subgroup,
then an argument similar to that used in the abelian case leads to a contradiction (\cref{prop:minnonabelnontriv}).
The deepest part of the proof is when $G$ is a monolithic group with nonabelian socle $M = T^\ell$.
In this case we have
    \[ T^\ell \unlhd G \leqslant \Aut(T)^\ell \rtimes \Sym(\ell) , \] 
    and $G_\alpha$ is a $2$-subgroup of $G$ having nilpotency class at most $2$ and exponent at most $4$.
    This includes the case $\ell=1$, i.e. $G$ being almost simple.
    We will proceed by showing that there always exists $m \in M$ such that $G_\alpha \cap G_\alpha^m =1$.
To do so, we first refine results of Zenkov \cite{Zenkov1996} and Burness--Huang \cite{BurHua2026}
on the intersections of Sylow $2$-subgroups in almost simple groups (\cref{lem:almostsimple}).
When $\ell >1$, we consider the natural action of $G_\alpha$ on $\{ 1,\ldots,\ell \}$
and use appropriate asymmetric colourings for permutation $2$-groups,
whose existence is proven in \cref{sec:Col}.
The idea of using distinct double-cosets to colour a direct product of groups was first used in \cite{Zenkov1996},
but the analysis here is much more delicate.
The details of the proof are given in \cref{sec:proof}.

\section{Preliminaries} \label{sec:basic}

\subsection{Graphs}

A graph $\Gamma$ is a pair $(V,E)$ where $V$ is a finite nonempty set of
{\em vertices} and $E$ is a set of unordered pairs of $V$, called {\em edges}.
An {\em $s$-arc} is an $(s+1)$-tuple of vertices with every two consecutive vertices adjacent and every three consecutive vertices pairwise distinct. In particular, a $1$-arc is also called an {\em arc}. 

We will also need a notion of {\em digraph}, which we define as a pair $\vGa = (V,A)$,
where $V$ is a finite nonempty set of vertices and $A$ is a set of {\em ordered} pairs of distinct vertices,
which we call {\em arcs}.
As above, an {\em $s$-arc} is an $(s+1)$-tuple of vertices 
such that every two consecutive vertices form an arc and
every three consecutive vertices are pairwise distinct.
If $(\alpha,\beta)$ is an arc of a digraph, then we say that $\beta$ is an {\em out-neighbour} of $\alpha$ and that
$\alpha$ is an {\em in-neighbour} of $\beta$.
The {\em out-valency} ({\em in-valency}, respectively) of a given vertex
is the number of its out-neighbours (in-neighbours, respectively).
The {\em neighbourhood} of $\alpha$ is $\Gamma(\alpha) = \{ \beta\in \V\Gamma \mid (\alpha,\beta)\in \A\Gamma\}$.

If $\vGa=(V,A)$, then the underlying graph of $\vGa$ is the graph $(V,E)$
with $E=\{ \{\alpha,\beta\} \mid (\alpha,\beta) \in A\}$. 
Note that, if $\vGa$ is an {\em orientation} (that is, $(\alpha,\beta) \in A$ implies $(\beta,\alpha)\not\in A$),
then there is a bijective correspondence between the arcs of $\vGa$
and the edges of the underlying graph.

\subsection{Groups}

Let $G$ be a finite group. For $x,g \in G$, let $x^g = g^{-1}xg$ and $[x,g] = x^{-1} x^g$.
If $X \subseteq G$, we write
$$ {\bf N}_G(X) = \{ g \in G : X^g = X \} ,  \qquad
{\bf C}_G(X) = \{ g \in G : x^g = x \text{ for all } x \in X \} , $$
and ${\bf Z}(G) = {\bf C}_G(G)$.
If $H$ is a subgroup of $G$, we write $H^G$ for the {\em normal closure} of $H$ in $G$.
The {\em derived subgroup} $[G,G]$ is the smallest normal subgroup of $G$
with respect to which the quotient group is abelian,
and the {\em Frattini subgroup} $\Phi(G)$ is the intersection of the maximal subgroups.

A finite group has trivial soluble radical if it has no nontrivial abelian normal subgroup,
and is {\em monolithic} if it has a unique minimal normal subgroup.
 The {\em socle} is the subgroup generated by the minimal normal subgroups.
Two minimal normal subgroups either coincide or commute,
so the socle is the direct product of the minimal normal subgroups.

If $G \leqslant \Sym(\Omega)$ is a permutation group and $\alpha \in \Omega$,
then we use $G_\alpha$ to denote the stabiliser of $\alpha$.
If $G$ stabilises $\Delta \subseteq \Omega$ setwise, then $G$ induces a (possibly nonfaithful) action on $\Delta$,
and we write $G^\Delta \leqslant \Sym(\Delta)$ for the corresponding permutation group.

Given a transitive permutation group $G \leqslant \Sym(\Omega)$,
a {\em block} is a nonempty subset $\Delta$ of $\Omega$ such that,
for every $g \in G$, either $\Delta^g = \Delta$ or $\Delta^g \cap \Delta = \emptyset$.
The $G$-orbit $\Sigma$ of a block forms a partition of $\Omega$, which is called a {\em system of imprimitivity}.
If $G$ admits a nontrivial system of imprimitivity, then $G$ is said to be {\em imprimitive},
otherwise it is said to be {\em primitive}.
An imprimitive permutation group $G$ having a block system $\Sigma$ containing a block $B$
embeds into the wreath product $G_B^B \wr_{\Sigma} G^{\Sigma}$.
In general, our notation for permutation groups follows \cite{DixonMortimer, MR3791829}.

Suppose that $G$ is a finite group acting transitively on a set $\Omega$.
For $x \in G$, let
\[
\Fix_\Omega(x) = \{ \omega \in \Omega \mid \omega^x = \omega \} ,
\]
the set of points fixed by $x$.
The {\em fixed-point ratio} is defined by
\begin{equation} \label{eq:fpr}
\mathrm{fpr}_\Omega(x)
   = \frac{|\Fix_\Omega(x)|}{|\Omega|} 
   = \frac{|x^G \cap G_\alpha|}{|x^G|} ,
\end{equation}
where the second equality is classical (see \cite{LiebeckSaxl1991}, for example).

A permutation group $G$ is said to be {\em semiregular} if $G_\alpha =1$ for all $\alpha \in \Omega$.
An orbit $\Delta$ of $G$ is {\em regular} if $|\Delta| = |G|$, or equivalently if $G_\alpha = 1$ for every $\alpha \in \Delta$.
We now give two easy results that produce regular orbits.
 
\begin{lemma}\label{lem:inflate}
Let $H \leqslant \Sym(\Delta)$ be transitive, let $\Sigma$ be a system of imprimitivity,
$\delta \in \Delta$, and let $\sigma \in \Sigma$ be the block containing $\delta$.
If the block stabiliser $H_\sigma$ has $\kappa$ regular orbits on $\Sigma$,
then $H_\delta$ has at least $\kappa |H_\sigma : H_\delta|^2$ regular orbits on $\Delta$.
\end{lemma}
\begin{proof}
Let $\Sigma_0$ be the subset of $\Sigma$ consisting of those blocks $\sigma' \in \Sigma$
such that $H_\sigma\cap H_{\sigma'} = 1$, that is, $\sigma'$ lies in a regular orbit for $H_\sigma$.
By hypothesis, $|\Sigma_0| = \kappa |H_\sigma|$.
Let
\[
\Delta_0 = \{ \delta' \in \Delta \mid \delta' \in \sigma' \text{ for some } \sigma' \in \Sigma_0 \}
          = \bigcup_{\sigma' \in \Sigma_0} \sigma'.
\]
If $\delta' \in \Delta_0$ with $\delta' \in \sigma' \in \Sigma_0$, then
\[
H_\delta \cap H_{\delta'} \leqslant H_\sigma \cap H_{\sigma'} = 1.
\]
Hence, $\delta'$ lies in a regular orbit for $H_\delta$. Moreover,
\[
|\Delta_0| =\kappa|H_{\sigma}||\sigma|= \kappa |H_\sigma| \cdot |H_\sigma : H_\delta|.
\]
Therefore, $H_\delta$ has at least
$|\Delta_0|/|H_\delta| = \kappa |H_\sigma : H_\delta|^2$
regular orbits on $\Delta$.
\end{proof}

\begin{lemma} \label{lem:threetimess}
    Let $P \leqslant \Sym(\Delta)$ be a permutation $2$-group, and suppose that
    \[ \sum_{\substack{h \in P \\ {\bf o}(h) = 2}} \mathrm{fpr}_\Delta(h) \le 1 - 3\frac{|P|}{|\Delta|}, \]
    where ${\bf o}(h)$ denotes the order of $h$.
    Then $P$ has at least three regular orbits on $\Delta$.
\end{lemma}
\begin{proof}
    Observe that the union
    \[ \bigcup_{h \in P,\, h \ne 1} \Fix_\Delta(h) \]
    is precisely the set of points of $\Delta$ not lying in a regular orbit for $P$.
    Moreover, since $\Fix_\Delta(h) \subseteq \Fix_\Delta(h^2)$ for all $h \in P$,
    it suffices to consider elements of order $2$.
    Consequently, $P$ has at least three regular orbits on $\Delta$ if and only if
    \[ \left| \bigcup_{\substack{h \in P \\ {\bf o}(h) = 2}} \Fix_\Delta(h) \right|
   \le |\Delta| - 3|P| . \]
   By a union bound, a sufficient condition for having three regular orbits is
   \[ \sum_{\substack{h \in P \\ {\bf o}(h) = 2}} \mathrm{fpr}_\Delta(h)  \le 1 - 3\frac{|P|}{|\Delta|} ,\]
   as desired.
\end{proof}

\subsection{Graphs and groups}

Let $\Gamma$ be a graph (or a digraph),
let $G \leqslant \Aut(\Gamma)$ be a group of automorphisms, and let $\alpha\in \V\Gamma$.
We say that $\Gamma$ is $(G,s)$-{\em arc-transitive} if $G$ acts transitively on the set of $s$-arcs of $\Gamma$.
When $G=\Aut(\Gamma)$, we omit the label $G$ and we simply say that $\Gamma$ is $s$-{\em arc-transitive}.

Let $G_\alpha^{\Gamma(\alpha)}$ be the permutation group induced by $G_\alpha$ on $\Gamma(\alpha)$.
Observe that a $G$-vertex-transitive graph $\Gamma$ is arc-transitive if and only if $G_\alpha^{\Gamma(\alpha)}$
is transitive on $\Gamma(\alpha)$,
and $(G,2)$-arc-transitive if and only if $G_\alpha^{\Gamma(\alpha)}$ is $2$-transitive on $\Gamma(\alpha)$.
For $r \geq 0$, let $G_\alpha^{[r]} \leqslant \Aut(\Gamma)$ be the automorphisms fixing a ball of radius $r$ around $\alpha$.
If $\Gamma$ is connected, then $G_\alpha^{[r]} =1$ for some $r$.

An edge- and vertex-transitive group of automorphisms $G$ of a connected graph $\Gamma$
that is not arc-transitive is called {\em half-arc-transitive}.
Note that in this case $G$ has two orbits on arcs, each orbit containing precisely one arc underlying each edge.
If $\Gamma$ is half-arc-transitive and $A$ is an orbit of $G$ on the arc-set of $\Gamma$,
then $(\V\Gamma,A)$ is an arc-transitive digraph, denoted $\vGa^{(G)}$, whose underlying graph is $\Gamma$.
In particular, if $\Gamma$ has valency $4$,
then the in-valence and out-valence of every vertex of $\vGa^{(G)}$ is $2$.

\subsection{Split and merge}\label{sec:splitMerge}

The operations of {\em splitting} and {\em merging} were introduced
in \cite[Constructions~7 and~11]{PotocnikSpigaVerret2013} to create a framework for translating results
from $3$-valent graphs with a perfect matching invariant under the action of their automorphism groups 
into results for $4$-valent graphs with a $2$-factor invariant under the action of their automorphism groups,
and vice versa.
The fact that these operators are the inverse of one another has been proved
in \cite[Theorem~12]{PotocnikSpigaVerret2013} and \cite[Theorem~2.9]{BarbieriGrazianSpiga2025},
up to a minor caveat.

This correspondence has a wide range of applications.
It has been used in \cite{PotocnikSpigaVerret2013} to construct censuses of graphs,
in \cite{PotocnikSpiga2021} to study fixed-point ratios,
in \cite{BarbieriGrazianSpiga2025} to investigate the asymptotic behaviour of the order of semiregular elements,
and in \cite{BarbieriZozaya} to construct cubic graphs of arbitrary even girth.
We refer the reader to \cite[Section~2.4]{BarbieriGrazianSpiga2025} and \cite[Section~4]{PotocnikSpigaVerret2013} 
for further details on these operations.

\smallskip\noindent\textsc{Splitting.}
Let $\Delta$ be a $4$-valent graph,
$G \leqslant \Aut(\Delta)$ vertex-transitive, edge-transitive, but not arc-transitive.
It follows that $G$ stabilises a $2$-factor $\mathcal{C}$ of $\Delta$.
We build a $3$-valent graph, $\mathrm{s}(\Delta,\mathcal{C})$, whose vertex-set is
\[ \V\mathrm{s}(\Delta,\mathcal{C}) = 
\left\{(\alpha,\mathbf{c}) \in \V\Delta \times \mathcal{C} \mid \alpha \in \V\mathbf{c} \right\} ,\]
and such that two vertices $(\alpha,\mathbf{c})$ and $(\beta,\mathbf{d})$ are adjacent
if either $\mathbf{c}$ and $\mathbf{d}$ are distinct and $\alpha= \beta$,
or $\mathbf{c}=\mathbf{d}$ and $\alpha$ and $\beta$ are adjacent in $\mathbf{c}=\mathbf{d}$.
The graph $\mathrm{s}\Delta$ is the {\em split} of the pair $(\Delta,\mathcal{C})$.
Note that $G$ acts faithfully and transitively on $\V\mathrm{s}\Delta$.

\smallskip\noindent\textsc{Merging.}
Let $\Gamma$ be a vertex-transitive but not arc-transitive cubic graph,
and let $G \leqslant \Aut(\Gamma)$ be vertex-transitive but not arc-transitive.
Then, for each $\gamma \in \V\Gamma$, there is a unique neighbour $\gamma'$ that is fixed pointwise by $G_\gamma$. Clearly, $(\gamma')' = \gamma$ and $G_\gamma = G_{\gamma'}$.
Moreover, the set
\[ \mathcal{M} = \{\{\gamma,\gamma'\} \mid \gamma \in \V\Gamma\} \]
is a complete matching of $\Gamma$.
The {\em merge} of $\Gamma$ is the graph ${\rm m}\Gamma$ whose vertex set is $\mathcal{M}$,
with two vertices $e_1, e_2 \in \mathcal{M}$ adjacent if and only if the corresponding edges of $\Gamma$ are at distance $1$. Note that ${\rm m}\Gamma$ can equivalently be obtained by contracting to a vertex the edges in $\mathcal{M}$.
Moreover, $G$ acts faithfully and transitively on $\V{\rm m}\Gamma$ and $E{\rm m}\Gamma$.

Two infinite families of cubic graphs have degenerate merged graphs,
namely the circular and M\"{o}bius ladders.

\begin{definition}\label{def:ladders}
    For any $n \ge 3$, a {\em circular ladder} is any graph isomorphic to the Cayley graph
    $$\Cay(\mathbb{Z}_n\times\mathbb{Z}_2,\{(0,1),(1,0),(-1,0)\}) . $$
    For any $n\ge 2$, a {\em M\"{o}bius ladder} is any graph isomorphic to the Cayley graph
    $$\Cay(\mathbb{Z}_{2n},\{1,-1,n\}).$$
\end{definition}

\subsection{Normal quotient method}\label{sec:normalQuot}

The notion of normal quotient of a graph, introduced in \cite[Section~4]{Praeger1993},
is the group-theoretic special case of the usual quotient construction in graph theory.
We record a version for digraphs.

\begin{definition}
Let $\vec \Gamma$ be a connected digraph and let $N \leqslant \Aut(\vec \Gamma)$.
The {\em normal quotient digraph} $\vec \Gamma/N$ is the digraph
whose vertices are the $N$-orbits on $\V\Gamma$,
with two distinct orbits $\alpha^N$ and $\beta^N$ joined by an arc $(\alpha^N,\beta^N)$ in $\vec \Gamma/N$
whenever there exist vertices $\alpha' \in \alpha^N$ and $\beta' \in \beta^N$
such that $(\alpha',\beta')$ is an arc of $\vec \Gamma$.
\end{definition}

We remark that, if $\vec \Gamma$ is connected, then so is $\vec \Gamma/N$.
Moreover, if $N$ is normalized by an overgroup $G \leqslant \Aut(\vec \Gamma)$,
then $G/N$ acts (possibly unfaithfully) on $\vec \Gamma/N$ as a group of automorphisms,
with vertex-stabiliser
\[(G/N)_{\alpha^N} = G_\alpha N / N .\]
If $G$ is vertex- or arc-transitive on $\vec \Gamma$, then $G/N$ is vertex- or arc-transitive on $\vec \Gamma/N$, respectively.
If $N$ is semiregular on $\V \vec\Gamma$, then in fact $(G/N)_{\alpha^N} \cong G_\alpha$.
If $G$ is arc-transitive, then the valency of $\vec \Gamma/N$ divides that of $\vec \Gamma$,
while if $G$ is only vertex-transitive, the valency of $\vec \Gamma/N$ is at most that of $\vec \Gamma$.

The analogous definition and statements apply to graphs, replacing arcs with edges.

\subsection{Structure of a vertex-stabiliser} \label{subsec:4}

For the rest of this section we specialise to connected vertex-transitive cubic graphs.
Let $\Gamma$ be a connected cubic graph,
$G \leqslant \Aut(\Gamma)$ a vertex-transitive group of automorphisms, and $\alpha \in \V\Gamma$.
A connectedness argument shows that $G_\alpha$ is a $\{2,3\}$-group (in particular, $G_\alpha$ is soluble).
In fact, if $x \in G_\alpha$, then $x^6 \in  G_\alpha^{[1]}$,
so $x^{6^r} \in  G_\alpha^{[r]}$,
and since $\Gamma$ is connected the order of $x$ is a divisor of $6^r$ for some $r$.
Moreover, it is clear that $G_\alpha$ is a $2$-group if and only if $G$ is not arc-transitive on $\Gamma$.

It is a fundamental theorem of Tutte \cite{Tutte1947,Tutte1959,Sims1967} that, if $G$ is arc-transitive,
then $|G_\alpha|$ divides $48$.
On the other hand, if $G$ is not arc-transitive, then $G_\alpha$ is a (possibly trivial) $2$-group.
We now use the work of Djokovi\'c \cite{Djokovic1980} to obtain an explicit group presentation for $G_\alpha$,
thus deducing crucial properties that will be of continuous use in this paper.

\begin{lemma} \label{lem:Galpha}
	Let $\Gamma$ be a vertex-transitive cubic graph,
	and let $G \leqslant \Aut(\Gamma)$ be not arc-transitive.
    If $\alpha \in \V\Gamma$, then $G_\alpha$ is a $2$-group of nilpotency class at most $2$ and exponent at most $4$.
\end{lemma}
\begin{proof}
    We can assume $|G_\alpha| \geq 8$.
    Observe that $G_\alpha$ fixes a unique neighbour $\alpha'$, and hence
    \( |G_{\{\alpha,\alpha'\}} : G_\alpha| = 2\).
    Furthermore, for $\beta \in \Gamma(\alpha) \setminus \{\alpha'\}$, 
    \[|G_\alpha : G_{\alpha,\beta}| = |G_{\{\alpha,\beta\}} : G_{\alpha,\beta}| = 2 . \]
    If we set
    \[ A_{-1} = G_{\{\alpha,\alpha'\}} , \quad A_0 = G_{\alpha,\beta}, \quad A_1 = G_{\{\alpha,\beta\}} , \]
    then, with the notation of~\cite{Djokovic1980}, the triple $(A_{-1},A_0,A_1)$ is an amalgam of degree $(4,2)$.
    Since $\Gamma$ is connected, $G = \langle A_{-1}, A_1 \rangle$ and the amalgam is faithful.
    As $|A_{-1} : A_0| = 4$, the action of $A_{-1}$ on the right cosets of $A_0$ yields either a dihedral group of order $8$ or a group of order $4$.
    In the latter case, $A_0$ is normal in $A_{-1}$, and hence $A_0$ is also normal in $G$.
    Since $G_\alpha$ is core-free in $G$, this forces $A_0 = 1$, which implies $|G_\alpha| = 2$.
    Since we are assuming $|G_\alpha| \ge 8$, the amalgam $(A_{-1},A_0,A_1)$ is of dihedral type.

    Let $m$ and $n$ be two integers such that $|G_\alpha| = 2^n$ and $m$ is minimal with the property $3m \ge 2n$.
    By~\cite[Theorem]{Djokovic1980},
    \begin{align*}
        G_\alpha &= \langle a_0, \ldots, a_{n-1} \rangle, \\
        a_i^2 &= 1 \qquad \text{for } 0 \le i \le n-1, \\
        [a_i,a_j] &= 1 \qquad \text{for } 0 \le |j-i| < m, 
    \end{align*}
and
$$
[a_i,a_j] = a_{n-m+i}^{\varepsilon(j-i,0)} a_{n-m+i+1}^{\varepsilon(j-i,1)} \cdots
a_{m+j-n}^{\varepsilon(j-i,j-i+2m-2n)} $$
for $j-i \ge m$, where each $\varepsilon(r,s) \in \{0,1\}$ satisfies the symmetry condition
\[
\varepsilon(r,s) = \varepsilon(r,r-s+2m-2n)
\]
for all $r,s$ with $m \le r \le n-1$ and $0 \le s \le r+2m-2n$.
Note that
\[ [G_\alpha,G_\alpha] = \langle a_{n-m}, \ldots , a_{m-1} \rangle \leqslant {\bf Z}(G_\alpha) ,\]
and hence $G_\alpha$ has nilpotency class at most $2$.
Since $G_\alpha$  is a group of class at most $2$ generated by involutions,
$G_\alpha$ has exponent at most $4$, which completes the proof.
\end{proof}

\subsection{Proof of \cref{cor:distinguishinCost}}\label{sec:corProof}

We first collect two lemmas which are needed for our study of distinguishing numbers.
We write $d \colon \V\Gamma^2 \to \mathbb{N}$ to denote the distance in the graph metric.

\begin{lemma}\label{lemma:K33}
    Let $\Gamma$ be a connected vertex-transitive cubic graph,
    and let $\alpha,\beta \in V\Gamma$ be distinct vertices.
    Suppose that $\Gamma(\alpha)=\Gamma(\beta)$.
    Then $\Gamma$ is isomorphic to ${\bf K}_{3,3}$.
\end{lemma}
\begin{proof}
    This requires a straightforward verification:
    we refer to \cite[Lemma~4.1]{PotocnikSpiga2021} for an explicit computation.
\end{proof}
 
\begin{lemma}\label{lemma:prism}
    Let $\Gamma$ be a cubic connected vertex-transitive graph,
    and let $\alpha,\beta,\gamma \in V\Gamma$ be three distinct vertices.
    Suppose that $\alpha$ and $\gamma$ are adjacent, that $|\Gamma(\alpha) \cap \Gamma(\beta)|=2$,
    and that $d(\alpha,\beta)=d(\beta,\gamma)=2$.
    Then $\Gamma$ is isomorphic to the circular ladder on $6$ vertices.
\end{lemma}
\begin{proof}
    To fix our notation, we write
    \[ \Gamma(\alpha)=\{\delta,\varepsilon,\gamma\},
    \quad \hbox{and} \quad
    \Gamma(\beta)=\{\delta,\varepsilon,\eta\}. \]
    Since \(d(\beta,\gamma)=2\), the vertex \(\gamma\) is adjacent to one of
    \(\delta,\varepsilon,\eta\). Interchanging \(\delta\) and \(\varepsilon\) if necessary, we may
    assume that $\gamma$ is adjacent to $\delta$ or \(\eta\). We consider these two cases in turn.

    Suppose first that $\gamma$ and $\delta$ are adjacent. Then \(\{\alpha,\delta,\gamma\}\) is a triangle.
    By vertex-transitivity, the vertex \(\beta\) is contained in a triangle as well. Since
    $\Gamma(\delta)= \{\alpha,\beta,\gamma\}$, this triangle must be
    \(\{\beta,\varepsilon,\eta\}\).
    Thus $\varepsilon$ and $\eta$ are adjacent. Now the six vertices
    \[ \alpha,\delta,\gamma,\beta,\varepsilon,\eta \]
    span a triangular prism, and all of them already have three neighbours.
    Since \(\Gamma\) is connected, these are all the vertices of \(\Gamma\),
    and hence $\Gamma$ is isomorphic to the circular ladder on $6$ vertices.
    Note that, in this case $\gamma$ is also adjacent to $\eta$.

    We now consider the case in which $\gamma$ is adjacent to $\eta$, and we may also assume that $\gamma$ is adjacent to neither \(\delta\) nor \(\varepsilon\).
    
    Since the only possible triangle through \(\alpha\) would use the edge \(\delta, \varepsilon\),
    and the same is true for \(\beta\), either \(\delta , \varepsilon \in E\Gamma\) or neither \(\alpha\) nor \(\beta\) lies in a triangle. If \(\delta , \varepsilon \in E\Gamma\), then \(\delta\) and \(\varepsilon\) are each contained in two triangles, namely \(\{\alpha,\delta,\varepsilon\}\) and \(\{\beta,\delta,\varepsilon\}\), whereas \(\alpha\) is contained in only one. This contradicts vertex-transitivity. Hence, the girth is $4$, as witnessed by the square $\{\alpha,\delta,\beta,\varepsilon\}$.

    Observe that $\{\alpha,\gamma,\eta,\beta,\delta\}$ and
    $\{\alpha,\gamma,\eta,\beta,\varepsilon\}$ are two distinct $5$-cycles.
    By vertex-transitivity, $\delta$ and $\varepsilon$ are contained in at least two $5$-cycles as well.
    Since the remaining $5$-cycles need to pass through $\alpha$ or $\beta$, the only way in which the count of $5$-cycles coincide
    through these points is that the remaining neighbours of $\delta$ and $\varepsilon$ are adjacent.
    In particular, three $5$-cycles pass through $\alpha$, $\beta$, $\delta$ and $\varepsilon$.
    We return our focus to the vertices $\gamma$ and $\eta$.
    The third $5$-cycle through them should pass through the remaining neighbours of both. The final contradiction arises because, if this holds, either the pair $\gamma$ and $\eta$ is not contained in a square, against vertex-transitivity, or the last $5$-cycle that we built contains a chord and, consequently, the girth of $\Gamma$ is $3$, which we already excluded.
\end{proof}

We are ready to prove \cref{cor:distinguishinCost} assuming the veracity of \cref{thrm:main}.

\begin{proof}[Proof of \cref{cor:distinguishinCost}]
If $\Gamma$ appears in \cref{tab:main}, then the result follows by direct inspection.
    If $\Gamma$ is a split Praeger--Xu graph, then its distinguishing number is $2$,
    while its distinguishing cost is unbounded \cite[Theorem~9.11]{Imrich2022}.

    Now let $\{\alpha,\beta\}$ be a base for the action of $\Aut(\Gamma)$ on $V\Gamma$.
    Choose a vertex $\gamma$ adjacent to $\alpha$ but not to $\beta$. Such a vertex exists in view of \cref{lemma:K33}.
    Observe that, unless $d(\alpha,\beta)=d(\gamma,\beta)$, the setwise stabiliser of $\{\alpha,\beta,\gamma\}$ is trivial,
    as the three pairwise distances between these vertices are distinct.
    Suppose therefore that $d(\alpha,\beta)=d(\gamma,\beta)$,
    and let $\delta$ be the neighbour of $\alpha$ lying on a geodesic segment from $\beta$ to $\alpha$.
    Then $d(\delta,\beta)=d(\alpha,\beta)-1$.
    If $d(\alpha,\beta)>2$, then it follows that the pairwise distances among $\alpha,\beta,\delta$ are distinct,
    and hence the setwise stabiliser of $\{\alpha,\beta,\delta\}$ is trivial.
    Otherwise, let $\varepsilon$ be the last neighbour of $\alpha$.
    If $\varepsilon$ is not a neighbour of $\beta$, then the setwise stabiliser of $\{\alpha,\beta,\varepsilon\}$ is trivial.
    Hence, we are left with assuming that $\Gamma(\alpha) \cap \Gamma(\beta) = \{\delta, \varepsilon\}$.
    \cref{lemma:prism} implies that $\Gamma$ is the circular ladder on $6$ vertices,
    where the distinguishing number is $2$ and the distinguishing cost is $3$.
    In all cases, colouring the three vertices in such a set with a colour different
    from all others yields a distinguishing colouring with two colours and cost $3$.
\end{proof}

\section{Arc-transitive graphs} \label{sec:AT}

The purpose of this section is to prove the following result,
which is the first piece of evidence towards the validity of \cref{thrm:main}.

\begin{proposition} \label{prop:arctran}
    Let $\Gamma$ be a connected arc-transitive cubic graph.
    If $\Gamma$ has base size greater than $2$, then $\Gamma$ is one of the graphs in \cref{tab:main}.
\end{proposition}
\begin{proof}
Let $G=\Aut(\Gamma)$ and $\alpha \in \V\Gamma$.
By Tutte's theorem we have $|G_\alpha| = 3 \cdot 2^{s-1}$ for some $s \in \{1,2,3,4,5\}$,
and $G$ acts regularly on the set of $s$-arcs of $\Gamma$.
Let $X=(\alpha = \alpha_0,\ldots,\alpha_s)$ be an $s$-arc of $\Gamma$.
If $s = 1$, then $G_{\alpha} \cap G_{\alpha_1} = 1$ and we are done.

Therefore, $s \ge 2$.
If $G_{\alpha} \cap G_{\alpha_s}\ne 1$,
then there exists a nontrivial automorphism $g \in G_{\alpha} \cap G_{\alpha_s}$.
Applying $g$ to $X$, we obtain the $s$-arc
\[
X^g = (\alpha, \alpha_1^g, \ldots, \alpha_{s-1}^g, \alpha_s).
\]
Since the pointwise stabiliser of $X$ is trivial, $X^g \neq X$.
Therefore, by concatenating $X$ and $X^g$, we obtain the closed walk 
\[
(\alpha, \alpha_1, \ldots, \alpha_s, \alpha_{s-1}^g, \ldots, \alpha_1^g, \alpha) ,
\]
and we deduce that $\Gamma$ has girth at most $2s$.

If $s = 2$, the girth of $\Gamma$ is at most $4$.
If $\Gamma$ has girth $3$, then the $2$-arc-transitivity of $\Gamma$ immediately implies
that $\Gamma$ is the complete graph ${\bf K}_4$.
If $\Gamma$ has girth $4$, then $2$-arc-transitivity and a brief analysis show
that either $\Gamma$ is ${\bf K}_{3,3}$ or the cube graph.

Suppose now that $s \in \{3,4\}$.
Conder and Nedela~\cite[Theorem~2.1]{ConderNedela2007} showed that if $\Gamma$ has girth at most $9$,
then it has at most $570$ vertices, and all such graphs are listed in~\cite[Table~1]{ConderNedela2007}.
The result then follows by checking, with the aid of a computer, the graphs arising in this classification.
The analogous classification when $s = 5$ and when the girth is at most $13$ is due to Morton~\cite{Morton1991}
(see also~\cite[page~760, lines~12--14]{ConderNedela2007}), and this case can be handled in a similar way.

To conclude, we observe that all graphs in \cref{tab:main} are distance-transitive.
\end{proof}

\begin{remark} \label{rem:prob1}
Although the proof of \cref{prop:arctran} is rather elementary,
it relies on deep results concerning cubic $s$-arc-transitive graphs and small girth.
This is one of the two major steps in the proof of \cref{thrm:main}
that we are not able to adapt to the case of $4$-valent arc-transitive graphs.
Indeed, the analogous classification of $4$-valent $s$-arc-transitive graphs with small girth is not only open,
but likely not meaningful, since there exist several classes of symmetric $4$-valent graphs with small girth.
Nevertheless, by following techniques similar to those used in \cite{ConderNedela2007,Morton1991},
it is possible to extract some information also in the $4$-valent case.
From a computational perspective, following the ideas in \cite{ConderNedela2007,Morton1991},
one can obtain an analogue of \cref{prop:arctran} for $4$-valent $2$-arc-transitive graphs
without requiring a classification of such graphs of small girth.
The only remaining obstruction arises in the case of $7$-arc-transitive $4$-valent graphs
where we could not conclude our analysis.
\end{remark}

\section{Praeger--Xu graphs and their split}\label{sec:PX}

This section introduces the Praeger--Xu graphs $\mathrm{C}(r,s)$ and their automorphism groups.
Although these graphs are $4$-valent, they play a pivotal role in the proof of \cref{thrm:main}.
This is due to the fact that most of the properties of the split Praeger--Xu graphs can be understood
from this family via the splitting and merging operations.

\subsection{Praeger--Xu graphs} \label{sub2}

The Praeger--Xu graphs were originally defined in \cite{PraegerXu1989} while studying graphs whose automorphism group contains a normal elementary abelian subgroup whose action is not semiregular.
Praeger--Xu graphs have been studied in detail by Gardiner, Praeger and Xu in
\cite{GardinerPraeger1994,Praeger1989HAT,PraegerXu1989}, and more recently in
\cite{BarbieriGrazianSpiga2022,BarbieriGrazianSpiga2023,JajcayPotocnikWilson2019,JajcayPotocnikWilson2022}.
Here, we introduce them through their directed counterparts defined in~\cite{Praeger1989HAT}.

Let $r$ be a positive integer with $r \ge 3$.
We define $\vec{\mathrm{C}}(r,1)$ to be the wreath product of an edgeless graph on two vertices by a directed cycle of length $r$.
In other words,
\[V\vec{\mathrm{C}}(r,1)=\mathbb{Z}_r\times\mathbb{Z}_2\]
with the out-neighbours of the vertex $(x, i)$ being $(x+1,0)$ and $(x+1,1)$.
We will identify the $(s-1)$-arc
\[(x,\epsilon_0)\to(x+1,\epsilon_1)\to\cdots\to(x+s-1,\epsilon_{s-1})\]
with the pair $(x; k)$ where $k=\epsilon_0\epsilon_1\ldots\epsilon_{s-1}$ is a string in $\mathbb{Z}_2$ of length $s$.

Now let $s$ be a positive integer with $s \ge 2$.
We let $V\vec{\mathrm{C}}(r,s)$ be the set of all $(s-1)$-arcs of $\vec{\mathrm{C}}(r, 1)$.
For every string $h$ in $\mathbb{Z}_2$ of length $s-1$, and for any $\epsilon\in \mathbb{Z}_2$,
we define the out-neighbours of $(x; \epsilon h) \in V\vec{\mathrm{C}}(r,s)$ to be $(x+1; h0)$ and $(x+1;h1)$.
Hence, the {\em Praeger--Xu graph} $\mathrm{C}(r,s)$ is defined as the underlying graph of $\vec{\mathrm{C}}(r,s)$.
Observe that $\mathrm{C}(r, s)$ is a connected $4$-valent graph with $r2^s$ vertices \cite[Theorem~2.8]{Praeger1989HAT}.

Let us now discuss the automorphisms of $\mathrm{C}(r, s)$.
 Every automorphism of $\vec{\mathrm{C}}(r,1)$ (or $\mathrm{C}(r, 1)$, respectively) acts naturally as an automorphism of $\vec{\mathrm{C}}(r, s)$ (or $\mathrm{C}(r, s)$, respectively) for every $s \ge 2$.
For $i \in \mathbb{Z}_r$, let $\tau_i$ be the transposition on $V\vec{\mathrm{C}}(r,1)$ swapping the vertices $(i,0)$
and $(i, 1)$ while fixing every other vertex.
This is  an automorphism of $\vec{\mathrm{C}}(r,1)$, and thus also of $\vec{\mathrm{C}}(r,s)$ for $s\ge 2$.
We set
\begin{equation}
    \label{eq:K}
    K = \langle \tau_i \mid i \in \mathbb{Z}_r\rangle,  
\end{equation}
and we observe that $K$ is isomorphic to $C_2^r$. Furthermore, let $\rho$ and $\sigma$ be the permutations
on $V\vec{\mathrm{C}}(r,1)$ defined by
\begin{equation}
    \label{eq:sigma}
    (i, \epsilon)^\rho = (i + 1, \epsilon) \quad \hbox{and} \quad (i, \epsilon)^\sigma = (-i, \epsilon).
\end{equation}
Then $\rho$ is an automorphism of $\vec{\mathrm{C}}(r, 1)$ of order $r$, and $\sigma$ is an involutory automorphism of $\mathrm{C}(r,1)$ (but not of $\vec{\mathrm{C}}(r,1)$). Observe that $\rho$ cyclically permutes the generators of $K$, while $\sigma$ is a permutation of such set of order $2$.
It follows that the group $\langle \rho, \sigma\rangle \cong D_{2r}$ normalises $K$.
We define
\begin{equation}
    \label{eq:H}
    H^+ = K \rtimes \langle \rho\rangle \quad \hbox{and} \quad H = K \rtimes \langle \rho, \sigma\rangle .
\end{equation}
Hence, for every $r \ge 3$ and $s \ge 1$,
\[ C_2^r \rtimes C_r \cong H^+ \leqslant \Aut(\vec{\mathrm{C}}(r,s))
\quad \textup{and} \quad 
C_2^r \rtimes D_{2r} \cong H \leqslant \Aut(\mathrm{C}(r, s))  \,.\]
Moreover, $H^+$ (or $H$, respectively) acts arc-transitively on $\vec{\mathrm{C}}(r,s)$ 
(or $\mathrm{C}(r,s)$, respectively) whenever $1 \le s \le r-1$.
If $r \neq 4$, then the groups $H^+$ and $H$ are, in fact,
the automorphism groups of $\vec{\mathrm{C}}(r,s)$ and $\mathrm{C}(r, s)$.

\begin{lemma}[\cite{PraegerXu1989} Theorem~2.13, and~\cite{Praeger1989HAT} Theorem~2.8]
    \label{bloodyhell}
    The automorphism group of a directed Praeger--Xu graph is \[\Aut(\vec{\mathrm{C}}(r, s)) = H^+ \,.\]
	If $r\ne 4$, then the automorphism group of a Praeger--Xu graph is
	\[\Aut(\mathrm{C}(r, s)) = H \,.\]
	Moreover,
	\begin{gather*}
		|\Aut(\mathrm{C}(4, 1)) : H| = 9 \,,\\
		|\Aut(\mathrm{C}(4, 2)) : H| = 3 \,,\\
		|\Aut(\mathrm{C}(4, 3)) : H| = 2 \,.
	\end{gather*}
\end{lemma}

\cref{bloodyhell} implies that $\C(r,s)$ is $2$-arc-transitive if and only if
$r=4$ and $s\in \{1,2\}$.
Let $\alpha$ be a vertex of $\vC(r,s)$ which as an $(s-1)$-arc of $\vC(r,1)$ starts in $(x,0)$ or $(x,1)$ for some $x\in \ZZ_r$.
Observe that
\[\Aut(\vC(r,s))_\alpha = \langle \tau_i \mid i\in \ZZ_r\setminus \{x,x+1, \ldots,x+s-1\}\rangle ,\]
showing that
\begin{equation}
\label{eq:KH+}
\begin{split}
   K &= \langle \Aut(\vC(r,s))_\alpha \mid \alpha \in \V\vC(r,s) \rangle
   \\&= \langle (H^+)_\alpha \mid \alpha \in \V\vC(r,s) \rangle 
   \\&= \langle K_\alpha \mid \alpha \in \V\vC(r,s) \rangle.   
\end{split}
\end{equation}

\begin{lemma} \label{lem:rs}
The  automorphism group of $\C(r,s)$ with $r\ge 3$ and $1\le s \le r-1$ has base size at most $2$ if and only if $2s> r$. Moreover, $K$ has base size greater than $2$ if and only if $2s< r$.
\end{lemma}
\begin{proof}
We start by considering the groups $H$ and $K$. 

Assume first $2s>r$.
Denote by ${\bf 0}$ the all-zero string in $\ZZ_2^s$.
Consider the vertices $\alpha=(0;{\bf 0})$ and $\beta=(s; {\bf 0})$.
We choose $g\in H_\alpha$ with $\beta^g=\beta$, and we aim to prove that $g$ is the identity.
We have
\[H_\alpha=\langle \sigma\rho^{s-1},\tau_{s},\ldots,\tau_{r-1}\rangle.\]
Thus
$g=(\sigma\rho^{s-1})^\varepsilon \tau_s^{\varepsilon_s}\cdots \tau_{r-1}^{\varepsilon_{r-1}}$ for some $\varepsilon,\varepsilon_s,\ldots,\varepsilon_{r-1} \in \{0,1\}$.
On one hand, if $\varepsilon=0$, then $\beta^g=\beta$ implies $(x,0)^{\tau_{x}^{\varepsilon_{x}}}=(s+x,0)$ for every $x\in \{s,\ldots,r-1\}$.
In particular, $\varepsilon_x=0$ and hence $g$ is the identity.
On the other hand, if $\varepsilon=1$, then $g$ maps $(s,0)$ to an element of the form $(r-1,j)$, which cannot be the tail of $\beta$ because $2s-1\not\equiv r-1\pmod r$.
Therefore, $g$ does not fix $\beta$.
In particular, $\{\alpha,\beta\}$ is a base for the action of $H$ on $\V\C(r,s)$ and hence it is also a base for the action of $K$.

Assume now $2s\le r$. Let $\alpha$ be as above, and let $\beta$ be an arbitrary vertex of $\C(r,s)$, that is, $\beta=(t;j)$ for some $t\in\ZZ_r$ and for some string $j=j_0\dots j_{s-1}\in\ZZ_2^s$.
Suppose $\{0,\ldots,s-1,t,t+1,\ldots,t+s-1\}$ is a proper subset of $\ZZ_r$.
Then there exists $y\in\ZZ_r \setminus \{0,\ldots,s-1,t,t+1,\ldots,t+s-1\}$.
Since $\tau_y$ fixes both $\alpha$ and $\beta$, $H_{\alpha}\cap H_\beta\ne 1$.
As $\tau_y \in K$, we also have $K_\alpha \cap K_\beta \ne 1$.
Therefore, we may suppose that $\ZZ_r=\{0,\ldots,s-1,t,t+1,\ldots,t+s-1\}$.
In particular, $2s\ge r$ and hence $r=2s$, because we are assuming $2s\le r$.
As $\ZZ_r=\{0,\ldots,s-1,t,t+1,\ldots,t+s-1\}$, we have $t=s$. For every $x\in \{0,\ldots,s-1\}$, let $\varepsilon_x=0$ if $j_x=j_{s-x-1}$ and $\varepsilon_x=1$ otherwise.
Consider $g=\sigma \rho^{s-1}\tau_{s}^{\varepsilon_0}\cdots\tau_{r-1}^{\varepsilon_{s-1}}$.
Observe that $g\in H_\alpha$.
For every $x\in \{0,\ldots,s-1\}$,
\begin{align*}
(s+x,j_x)^{g}&=(-s-x,j_x)^{\rho^{s-1}\tau_{s}^{\varepsilon_0}\cdots\tau_{r-1}^{\varepsilon_{s-1}}}
\\&=(-x-1,j_x)^{\rho^{(s-1)}\tau_{s}^{\varepsilon_0}\cdots\tau_{r-1}^{\varepsilon_{s-1}}}
\\&=(s+(s-x-1),j_x)^{\tau_{r-x-1}^{\varepsilon_{s-x-1}}}
\\&=(s+s-x-1,j_{s-x-1}).
\end{align*}
Thus $g$ fixes also $\beta$, and hence $H_\alpha\cap H_\beta\ne 1$.
In particular, the base size for the action of $H$ on $\V\C(r,s)$ is at least $3$.
On the other side, we have $K_\alpha\cap K_\beta=1$ and hence, when $r=2s$, $K$ has base size $2$.

Summing up, $H$ has base size $2$ if and only if $2s>r$, and $K$ has base size $2$ if and only if $2s\ge r$.
Now the result follows from \cref{bloodyhell}, except when $(r,s)=(4,3)$. 
When $(r,s)=(4,3)$, it can be verified with a computer that $\Aut(\C(4,3))$ has base size $2$.
\end{proof}

The following is a crucial property of the Praeger--Xu graphs.

\begin{lemma}{\rm (see~\cite[Theorem~1]{PraegerXu1989} and~\cite[Theorem 2.9]{Praeger1989HAT})} \label{cor:3.4}
Let $\Gamma$ be a connected $4$-valent graph
and let $G$ be an edge- and vertex-transitive group of automorphisms of $\Gamma$.
If $G$ has an abelian normal subgroup which is not semiregular on $V\Gamma$,
then $\Gamma\cong \C(r,s)$ with $r\ge 3$ and $1\le s\le r-1$. 
\end{lemma}

\subsection{Covers of Praeger--Xu graphs}

The following is the main result of this section.

\begin{proposition} \label{prop:vCcov}
Let $\Gamma$ be a connected $4$-valent graph,
let $G\leqslant \Aut(\Gamma)$ be vertex- and edge-transitive with base size greater than $2$.
Let $N$ be a minimal normal $2$-subgroup of $G$ such that the quotient
$\Gamma/N$ is isomorphic to $\C(r,s)$ for some $r$ and $s$.
Then $\Gamma$ is isomorphic to $\C(r',s')$ for some $r'$ and $s'$ with $1\le s'\le r'-1$ and $2s'\le r'$.
\end{proposition}

The proof of \cref{prop:vCcov} is quite involved and relies on a careful analysis of earlier work on covers of Praeger--Xu graphs. Since the statement resembles that of~\cite[Lemma~2.3]{PotocnikSpiga2021}, we follow that paper closely, in particular in the general strategy of the argument and in the initial setup.
We also make use of ideas developed in~\cite{PotocnikSpigaVerret2015,PotocnikVerret2010},
where certain covers of Praeger--Xu graphs are investigated further.
The two proofs, however, diverge at a later stage.
In particular, the final part of our construction, involving two specially chosen vertices,
is instead inspired by the graphs $\Gamma_t^\pm$ introduced in~\cite{PotocnikSpigaVerret2011}.
The automorphism groups of $\Gamma_t^\pm$ have base size $2$,
which motivated our particular choice of the vertex $\beta'$ below.
In fact, a direct computation in $\Aut(\Gamma_t^\pm)$ shows that, up to symmetry, this choice of $\beta'$ is essentially the only possible one.

\begin{proof}[Proof of \cref{prop:vCcov}]
Observe that once we have shown that $\Gamma \cong \C(r',s')$ for some integers $r'$ and $s'$, the condition $2s' \le r'$ follows from \cref{lem:rs} because the base size of $G$ is greater than $2$.
Thus, we argue by contradiction and assume that $\Gamma$ is not isomorphic to $\C(r',s')$ for any integers $r'$ and $s'$.

Since $\Gamma/N \cong \C(r,s)$, we have $G/N \leqslant \Aut(\C(r,s))$.
Recall that $H \leqslant \Aut(\C(r,s))$, where $H$ is defined in (\ref{eq:H}).
We split the discussion into three cases, depending on whether $G/N \leqslant H^+$,
or $G/N \leqslant H$ but $G/N \nleqslant H^+$, or $G/N \nleqslant H$.

\smallskip
\noindent\textsc{Suppose that $G/N \leqslant H^+$.} 
As $H^+$ preserves the orientation of $\vC(r,s)$, the group $G/N$ is not arc-transitive, and thus neither is $G$.
Therefore, $G$ acts half-arc-transitively on $\Gamma$.
Let $\vGa=\vGa^{(G)}$ be one of the two digraphs induced by the half-arc-transitive action of $G$.
Then $\vGa$ is an arc-transitive digraph of in- and out-valency $2$, and thus,
in view of \cref{cor:3.4}, 
\begin{equation} \label{eq:no}
\hbox{every abelian normal subgroup of } G \hbox{ acts semiregularly on } \V\vGa.
\end{equation}

Note that $\vGa/N$ is isomorphic to $\vC(r,s)$.
By identifying $\vGa/N$ with $\vC(r,s)$, we may consider the automorphisms $\tau_i, \rho \in \Aut(\vGa)$
and the group $K = \langle \tau_i \mid i\in \ZZ_r\rangle \cong \C_2^r$ (as in \cref{sub2}).
Observe that, for every vertex $\alpha \in \V\Gamma$,
$(G/N)_{\alpha^N} \leqslant K_{\alpha^N}$,
and hence $(G/N)_{\alpha^N} = (G/N) \cap K_{\alpha^N}$.
Define
\begin{equation*}
   E = \langle G_\beta \mid \beta \in \V\vGa\rangle,
\end{equation*}
and observe that, by the transitivity of $G$ on $\V\Gamma$,
 $E$ coincides with the normal closure $(G_\alpha)^G$.
Using (\ref{eq:KH+}), we deduce
\begin{equation}
    \label{eq:propE}
\begin{split}
   EN/N &= (G_\alpha N)^G/N
   \\&= (G_\alpha N/N)^{G/N}
   \\&= ((G/N)_{\alpha^N})^{G/N}
   \\&= ((G/N) \cap K_{\alpha^N})^{G/N}
   \\&= (G/N) \cap K.
\end{split}
\end{equation}
By minimality of $N$, it follows that either $N \leqslant E$ or $N \cap E = 1$.
If the latter holds, then $E \cong EN/N \leqslant K$. In particular, $E$ is an abelian normal subgroup of $G$ that does not act semiregularly on $\V\vGa$, contradicting (\ref{eq:no}).

Thus, we must assume that $N \leqslant E$.
Then $E/N = EN/N = (G/N) \cap K$.
Hence, $E/N$ is an elementary abelian $2$-group, implying that $E$ is a $2$-group.
Moreover, since
\[
E_\alpha = E_\alpha /(N\cap E_\alpha) \cong (E_\alpha N/N) = (E/N)_{\alpha^N} \leqslant K_{\alpha^N},
\]
we obtain that
\begin{equation}
\label{eq:Euab}
E_\alpha \hbox{ is an elementary abelian } 2\hbox{-group, for every } \alpha\in \V\Gamma.
\end{equation}

Being characteristic in $E$, both $[E,E]$ and $\Phi(E)$ are normal in $G$.
Moreover $[E,E] \leqslant \Phi(E) \leqslant N$,
and by the minimality of $N$, we have either $[E,E] = 1$ or $[E,E] =N$.
In the first case $E$ is abelian, which contradicts (\ref{eq:no}).
Therefore,
\begin{equation*}
\label{eq:EE}
   [E,E] = \Phi(E) = N.
\end{equation*}

Moreover, since $E$ and $N$ are $2$-groups,
the action of $E$ on $N\setminus\{1\}$ by conjugation must have at least one fixed point,
implying that $N$ intersects the center $\Z E$ non-trivially.
The minimality of $N$ then implies that $N \leqslant \Z E$. 
If $\Z E_\alpha \ne 1$, then $\Z E$ is a normal abelian $2$-subgroup of $G$ that does not act semiregularly on $\V\Gamma$, which contradicts (\ref{eq:no}), showing that 
\begin{equation}
    \label{eq:Z}\Z E\hbox{ acts semiregularly on }\V\Gamma.
\end{equation}

We will now set up a standard notation typically used when studying 
the structure of a vertex-stabiliser $G_\alpha$ in an arc-transitive digraph of out-valence $2$
(see, for example,~\cite[Section~2.3]{PotocnikVerret2010}).
Let $t$ be the largest integer such that $G$ acts transitively on the $t$-arcs of $\vGa$.
Note that $G$ must act regularly on the set of all $t$-arcs of $\vGa$ and that, for every $\alpha\in  \V\Gamma$,
$t$ is the largest integer such that $G_\alpha$ acts transitively on the $t$-arcs starting at $\alpha$.
Let $y$ be an element of $G$ such that $(\alpha^y,\alpha)$ is an arc of $\vGa$ and let 
\begin{equation*}
\label{eq:vi}
   \alpha_i = \alpha^{y^{-i}} \hbox{ for } i\in \ZZ.
\end{equation*}
Note that, for every $i\ge 0$, the $(i+1)$-tuple $(\alpha_0,\alpha_1, \ldots, \alpha_i)$ is an $i$-arc of $\vGa$.

Since $E = (G_\alpha)^G$, we have $E_\alpha = G_\alpha$ for every $\alpha \in \V\Gamma$.
Moreover, by the half-arc-transitivity of $G$, we have $\vGa/E \cong \vC(r,s)/K \cong \vec{\C}_r$.
Therefore, $Ey \in G/E$ acts as a one-step rotation of $\vGa/E$, implying that $G = E\langle y\rangle$.
Consider the stabiliser
$G_{(\alpha_0, \ldots, \alpha_i)}$ of the $i$-arc $(\alpha_0,\ldots,\alpha_i)$ in $G$.
Since $\alpha_0 = \alpha$ and $G_\alpha = E_\alpha$, we see that
$G_{(\alpha_0, \ldots, \alpha_i)} = E_{(\alpha_0, \ldots, \alpha_i)}$.
By the definition of $t$, it follows that $G_{(\alpha_0, \ldots, \alpha_{t})}$ is trivial
and that $G_{(\alpha_0, \ldots, \alpha_{t-1})}$ is cyclic of order $2$.
Let $x_0$ be its unique nonidentity element, that is,
the automorphism in $G$ that fixes the $(t-1)$-arc $(\alpha_0,\alpha_1, \ldots, \alpha_{t-1})$
but moves the vertex $\alpha_t$.
For $i\ge 1$, define
\begin{equation*}
\label{eq:xi}
x_i = x_0^{y^{i}} \quad \hbox{ and } \quad E_i = \langle x_0, \ldots, x_{i-1} \rangle, \quad E_0 = 1.
\end{equation*}
Following~\cite[Section~2.3]{PotocnikVerret2010}, we have
\begin{equation*}
\label{eq:Ei}
\hbox{for every } i\in\{0, \ldots, t\}, \qquad
  E_i = G_{(\alpha_0, \ldots, \alpha_{t-i})} \quad \hbox{ and } \quad |E_i| = 2^i.
\end{equation*}
Moreover, again from~\cite[Section~2.3]{PotocnikVerret2010}, there exists a positive integer $e$ such that
\begin{itemize}
   \item $e$ is the smallest integer such that $E_{t+e} = E_{t+e+1}$;
   \item $e$ is the smallest integer such that $E_{t+e} = E$.
\end{itemize}

Recall that $E/N = (G/N) \cap K$ is an elementary abelian $2$-group. We claim that
\begin{equation}
\label{eq:E/N}
 |E/N| = 2^{t+e} \quad \hbox{ and } \quad E/N = \langle Nx_0, Nx_1, \ldots, Nx_{t+e-1} \rangle.
\end{equation}
Let $e' \in \mathbb{N}$ be minimal with $E_{t+e'}N = E_{t+e'+1}N$.
As $E_{t+e} = E_{t+e+1}$, we have $e' \le e$.
Since $E_i^y = \langle x_1, \ldots, x_i \rangle$ for every $i$,
it follows that
\begin{align*}
        E_{t+e'+2}N & = \langle E_{t+e'+1}N, (E_{t+e'+1}N)^y \rangle \\
        &= \langle E_{t+e'}N, (E_{t+e'}N)^y \rangle \\
        &= E_{t+e'+1}N \\
        &= E_{t+e'}N.
    \end{align*}
Thus, by induction we see that $E = E_{t+e}N = E_{t+e'}N$.
Since $N = \Phi(E)$, the set of non-generators of
$E$, it follows that $E_{t+e'} = E$, and thus $e = e'$. 
In particular, $e$ is the smallest integer such that $E_{t+e}N = E_{t+e+1}N$.
Hence
\begin{equation*}
 N = E_0N < E_1N < E_2N < \cdots < E_{t+e-1}N < E_{t+e}N = E.
\end{equation*}
In particular, $|E| \ge 2^{t+e}|N|$, and thus $|E/N| \ge 2^{t+e}$.
On the other hand, we obtain that
$E/N = E_{t+e}/N = \langle Nx_0, Nx_1, \ldots, Nx_{t+e-1} \rangle$.
Since $E/N$ is elementary abelian, we see that $|E/N| \le 2^{t+e}$.
This proves the claims in (\ref{eq:E/N}).

Recall from (\ref{eq:Z}) that $\Z E$ acts semiregularly on $\V\Gamma$, and that,
as stated in (\ref{eq:Euab}), $E_\alpha = E_t = \langle x_0, \ldots, x_{t-1} \rangle$ is abelian.
It follows that
\[
E_t^{y^{t-1}} = \langle x_{t-1}, \ldots, x_{2t-2} \rangle
\]
is also abelian. Therefore $x_{t-1}$ is central in
\[
\langle E_t, (E_t)^{y^{t-1}} \rangle
= \langle x_0, \ldots, x_{2t-2} \rangle
= E_{2t-1}.
\]
Since $x_{t-1} \in E_\alpha$ and $\Z E \cap E_\alpha = 1$, we get $E_{2t-1} < E = E_{t+e}$, and hence $2t-1 < t+e$, from which it follows that
\begin{equation} \label{eq:mygod1}
e \ge t.
\end{equation} 

Let $\gamma = \alpha^{y^t}$.
We prove that $G_\alpha \cap G_\gamma = 1$: this would imply that $G$ has base size $2$,
contradicting our assumption on $G$.
By our previous discussion, we have
\[
G_\alpha = E_\alpha = E_t \quad \hbox{and} \quad 
G_\gamma =G_{\alpha^{y^t}}= E_t^{y^t} = \langle x_t, \ldots, x_{2t-1} \rangle.
\]
Aiming for a contradiction, assume that $G_\alpha$ and $G_\gamma$ intersect non-trivially. Let $g \in G_\alpha \cap G_\gamma$ be a nontrivial automorphism. Note that $g\in G_\gamma$ and hence
\[
g = x_t^{\varepsilon_t} x_{t+1}^{\varepsilon_{t+1}} \cdots x_{2t-1}^{\varepsilon_{2t-1}} \in E_t,
\]
for some $\varepsilon_t, \ldots, \varepsilon_{2t-1} \in \{0,1\}$.
Since $g$ is nontrivial, there exists $t' \in \{t, \ldots, 2t-1\}$ with $\varepsilon_{t'} = 1$
and with $\varepsilon_{t''} = 0$ for each $t'' > t'$.
This implies
\[
x_{t'} \in \langle E_t, g, x_t, \ldots, x_{t'-1} \rangle
= \langle E_t, x_t, \ldots, x_{t'-1} \rangle
= \langle x_0, \ldots, x_{t'-1} \rangle
= E_{t'}.
\]
Thus,
\[
E_{t+(t'-t)+1} = E_{t'+1} = E_{t'} = E_{t+(t'-t)},
\]
and the minimality of $e$ implies $e \le t' - t$.
As $t' \le 2t-1$, we deduce that $e \le t-1$, which contradicts the inequality in (\ref{eq:mygod1})
and concludes the analysis of the case $G/N \leqslant H^+$.

Before dealing with the case $G/N \nleq H^+$,
we need to make one additional observation in the case $G/N \leqslant H^+$,
which will be crucial later in the argument.
We claim that
\begin{equation}\label{eq:newClaim}
   r \ge t + e.
\end{equation}
We argue by contradiction and suppose that $r < t + e$.
In particular, $0\le t+e-r-1\le r+e-1$ and
\[
E = E_{t+e} = \langle x_0, \ldots, x_{t+e-r-1}, \ldots, x_{t+e-1} \rangle.
\]
Recall that $G/E = \langle Ey \rangle$ is a cyclic group of order $r$.
Hence, on the one hand, $y^r \in E$, and, on the other hand, $x_{t+e-1} = x_{t+e-r-1}^{y^r}$. Hence,
\begin{align*}
E &= E_{t+e} = \langle x_0, \ldots, x_{t+e-2}, x_{t+e-1} \rangle
\\&= \langle x_0, \ldots, x_{t+e-2}, x_{t+e-r-1}^{y^r} \rangle
   = \langle E_{t+e-1}, x_{t+e-r-1}^{y^r} \rangle
\\&= \langle E_{t+e-1}, x_{t+e-r-1}[x_{t+e-r-1}, y^r] \rangle
\\&= \langle E_{t+e-1}, [x_{t+e-r-1}, y^r] \rangle
\\&= \langle E_{t+e-1}, \Phi(E) \rangle
\\&= E_{t+e-1},
\end{align*}
where we use the fact that $\Phi(E)$ is the set of non-generators of $E$ for the last equality.
In particular, the equality $E_{t+e} = E_{t+e-1}$ contradicts the minimality of $e$, thus proving the claim.

\smallskip
\noindent\textsc{Suppose that $G/N \leqslant H$ but $G/N \nleqslant H^+$.}
Let $G^+$ be the subgroup of $G$ of index $2$, with $G^+/N \leqslant H^+$.
The argument above shows that $G^+$ has base size $2$ and, in fact, we may take the vertices $\alpha$ and $\gamma = \alpha^{y^t}$ to witness that $G_\alpha^+ \cap G_\gamma^+ = 1$.
If $G_\alpha \cap G_\gamma = 1$, then we contradict the hypothesis that $G$ has base size greater than $2$.
Therefore, $G_\alpha \cap G_\gamma \ne 1$.
As $|G : G^+| = 2$, we deduce that $|G_\alpha \cap G_\gamma| = 2$.
Using the notation in (\ref{eq:K}) and (\ref{eq:sigma}),
we have $(G_\alpha \cap G_\gamma)N/N = \langle k\sigma \rangle$, for some $k \in K$.
As in the case $G/N \leqslant H^+$, we can define the subgroup $E$ of $G$ with respect to $G^+$.

As $\Gamma/E \cong \C(r,s)/K$ is a cycle of length $r$, and since $\sigma$ acts as a reflection of this cycle,
we deduce that elements of the form $k\sigma$ can fix vertices only in the following cases: 
\begin{itemize}
\item if $r$ is odd, then such elements fix vertices in $\alpha^E$; 
\item if $r$ is even, then they fix vertices in $\alpha^E \cup \alpha'^{E}$,
where $\alpha'^E$ is the vertex opposite (that is, at distance $r/2$) to $\alpha^E$ in the quotient graph $\Gamma/E$. 
\end{itemize}

Since $\gamma = \alpha^{y^t}$, and since $r \ge t+e$ and $e \ge t$ (see (\ref{eq:mygod1}) and (\ref{eq:newClaim})),
it follows that the only possible way for $G_\alpha \cap G_\gamma \ne 1$ to occur is when $e = t = r/2$.
Therefore, we assume this equality in what follows.

As $e = t$ and $\gamma = \alpha^{y^t}$, we have
\[
E_\alpha = \langle x_0, \ldots, x_{t-1} \rangle, \quad
E_\gamma = \langle x_t, \ldots, x_{2t-1} \rangle, \quad
E = \langle E_\alpha, E_\gamma \rangle, \quad
|E : N| = 2^{2t}.
\]
Now let $\beta = \alpha^{y^{t-1}}$.
Since $\Gamma/E$ is a cycle of length $r = 2t$, we deduce that, for every $\beta' \in \beta^E$,
we have $G_\alpha \cap G_{\beta'} \leqslant E$, and hence $G_\alpha \cap G_{\beta'} = E_\alpha \cap E_{\beta'}$.
If $\beta' = \beta^{x_{2t-1}}$, then we claim that $G_\alpha \cap G_{\beta'} = 1$.

Let $n = [x_{t-1}, x_{2t-1}]$. As $E/N$ is abelian, we have $n \in N$.
Suppose that $n = 1$. Then $x_{t-1}$ is centralized by $x_{2t-1}$.
As
$x_{t-1} \in \langle x_0, \ldots, x_{t-1} \rangle = E_\alpha$
and
$x_{t-1} \in \langle x_{t-1}, \ldots, x_{2t-2} \rangle = E_\beta$,
and since $E_\alpha$ and $E_\beta$ are abelian by (\ref{eq:Euab}),
we deduce that $x_{t-1}$ is centralized by $\langle x_0, \ldots, x_{2t-2} \rangle$.
Since $x_{t-1}$ is centralized by $x_{2t-1}$, we conclude that
\[
x_{t-1} \in\Z{\langle x_0,\ldots,x_{2t-1}\rangle}\cap E_\alpha= \Z E \cap E_\alpha = 1,
\]
by (\ref{eq:Z}). This contradiction shows that $n$ is nontrivial.

Observe that
\begin{align*}
E_{\beta'}&=E_{\beta^{x_{2t-1}}} = E_\beta^{x_{2t-1}}
\\& = \langle x_{t-1}^{x_{2t-1}},x_t^{x_{2t-1}},\ldots,x_{2t-2}^{x_{2t-1}}\rangle
\\& = \langle x_{t-1}^{x_{2t-1}},x_t,\ldots,x_{2t-2}\rangle
\\& = \langle x_{t-1}n,x_{t},\ldots,x_{2t-2}\rangle.
\end{align*}
Thus $E_{\beta'}$ equals $E_\beta$ modulo $N$.
Since the intersection of $E_\alpha N$ and $E_\beta N$ is $\langle x_{t-1}\rangle N$, we deduce
\[E_\alpha \cap E_{\beta'} \leqslant \langle x_{t-1}\rangle N. \]
As 
\[x_{t-1}n=x_{t-1}[x_{t-1},x_{2t-1}]=x_{t-1}^{x_{2t-1}}\in E_{\beta'} ,\]
the modular law implies
\[E_{\beta'}\cap \langle x_{t-1}n\rangle N=\langle x_{t-1}n\rangle(E_{\beta'}\cap N)=\langle x_{t-1}n\rangle N_{\beta'}=\langle x_{t-1}n\rangle.\]
In turn, this yields $E_\alpha \cap E_{\beta'} \leqslant \langle x_{t-1}n\rangle$.
If $E_\alpha \cap E_{\beta'}$ contains a nontrivial automorphism, then $x_{t-1}n\in E_\alpha$.
However, since $x_{t-1}\in E_\alpha$, we get $n\in E_\alpha\cap N=1$, and hence $n=1$, which is again a contradiction. 

Summing up, in this case, $G$ cannot have base size larger than $2$.

\smallskip
\noindent\textsc{Suppose that $G/N \nleqslant H$.} 
As $G/N \leqslant \Aut(\Gamma/N) = \Aut(\C(r,s))$, \cref{bloodyhell} implies that $r = 4$.
At this point, the proof becomes entirely computational.
For each subgroup $X$ of $\Aut(\C(4,s))$ with $X \nleqslant H$,
we determine all irreducible $X$-modules $N$ over a field of characteristic $2$.
Then, for each such choice of $X$ and $N$, we compute all possible extensions of $N$ by $X$.
These groups constitute our candidate groups $G$.
For each such group $G$, we determine whether it acts on a connected cubic graph and has base size greater than $2$.
No exceptional cases arise from this analysis.
\end{proof}

\subsection{Split Praeger--Xu graphs} \label{sec:subS}

The family of the {\em split Praeger--Xu graphs},
featured in \cref{thrm:main}, is obtained from the Praeger--Xu graphs via the splitting operation,
which we explained in \cref{sec:splitMerge}.
By \cite[Lemma~2.12]{BarbieriGrazianSpiga2025},
we have $\Aut(\Spl\C(r,s)) = H$, where $H$ is defined in (\ref{eq:H}).
It follows that
$$ \Aut(\Spl\C(r,s)) \cong \Aut(\C(r,s)) \cong C_2^r \rtimes D_{2r} . $$
Moreover, for $r \neq 4$, the stabiliser of a vertex is an elementary abelian $2$-group.

\begin{lemma} \label{bloodyhel}
    The automorphism group of the split Praeger--Xu graph $\Spl\C(r,s)$,
    with $r \ge 3$ and $1 \le s \le r-1$, has base size greater than $2$ if and only if $2s < r$.
\end{lemma}
\begin{proof}
We use the notation established above.
Observe that $H_{\alpha^+}$ is a subgroup of index $2$ in $H_\alpha$. In fact,
$H_{\alpha^+} = H_\alpha \cap K = K_\alpha$, where $K$ is defined in (\ref{eq:K}). In particular, by \cref{lem:rs}, we deduce that the automorphism group of $\Spl\C(r,s)$ has base size greater than $2$ if and only if $2s < r$.
\end{proof}

We conclude this section with two technical lemmas.
The second one shows the relevance of the split Praeger--Xu graphs
and the ladders described in \cref{def:ladders} while dealing with quotients of cubic graphs.

\begin{lemma}[\cite{BarbieriGrazianSpiga2025}, Lemma~2.3]\label{lem:bigbigbig}
Let $\Gamma$ be a connected cubic graph, $\alpha \in \V\Gamma$,
$G \leqslant \Aut(\Gamma)$ vertex-transitive, and let $N$ be a semiregular normal subgroup of $G$.
Suppose that the action of $G_\alpha$ on $\Gamma(\alpha)$ is cyclic of order $2$,
and that the quotient $\Gamma /N$ is a cycle of length $r \ge 3$.
Let $K$ be the kernel of the action of $G$ on the $N$-orbits on $\V\Gamma$.
Then either
\begin{enumerate}
\item\label{emui1} $G_\alpha$ has order $2$ and $K_\alpha=1$, or
\item\label{emui2} $r$ is even and $G_\alpha=K_\alpha$ is an elementary abelian $2$-group of order at most $2^{r/2}$.
\end{enumerate}
\end{lemma}

\begin{lemma}[\cite{BarbieriGrazianSpiga2025}, Lemma~2.14] \label{lem:GNc}
Let $\Gamma$ be a connected vertex-transitive cubic graph, $\alpha\in \V\Gamma$,
$G \leqslant \Aut(\Gamma)$ vertex-transitive, and let $N$ be a minimal normal subgroup of $G$.
Suppose that the action of $G_\alpha$ on $\Gamma(\alpha)$ is cyclic of order $2$,
that $N$ is a $2$-group and $\Gamma/N$ is a cycle of length at least $3$.
Then $\Gamma$ is isomorphic to a circular ladder, or to a M\"{o}bius ladder,
or to $\Spl\C(r,s)$ for some $r\ge 3$ and $s\le r-1$.
\end{lemma}

\section{Abelian minimal normal subgroups} \label{sec:abelminnor}

If $\Gamma$ is not arc-transitive,
then the proof of \cref{thrm:main} proceeds by induction on the number of vertices.
In this section we handle the case where $G \leqslant \Aut(\Gamma)$ has a nontrivial abelian normal subgroup.

\begin{proposition} \label{prop:minabel} 
Let $\Gamma$ be a connected vertex-transitive cubic graph.
Let $G \leqslant \Aut(\Gamma)$ act transitively on $\V\Gamma$ and let $N$ be an abelian minimal normal subgroup of $G$.
Suppose that, if $\Gamma/N$ has valency $3$, then \cref{thrm:main} holds for $\Gamma/N$.
Then \cref{thrm:main} holds for $\Gamma$.
\end{proposition}
\begin{proof}
By \cref{prop:arctran}, we can assume $G$ is not arc-transitive.
If $\alpha \in \V\Gamma$, then, by \cref{lem:Galpha},
$G_\alpha$ is a $2$-group of nilpotency class at most $2$ and exponent at most $4$.
Moreover, we may suppose that $|G_\alpha| \ge 4$, and we divide the proof according to the valency of $\Gamma/N$.

\smallskip
\noindent\textsc{$\Gamma/N$ has valency $0$.}
In this case $N$ acts transitively on $\V\Gamma$, and, by Frattini's Argument, $G = N G_\alpha$.
Since $N$ is abelian, it acts regularly on $\V\Gamma$.
Thus, $\Gamma$ is a Cayley graph over $N$.
As $\Gamma$ is cubic, $|\V\Gamma|$ is even and hence $N$ is an elementary abelian $2$-group.
Since both $G_\alpha$ and $N$ are $2$-groups, so is $G$.
As $N$ is minimal normal in $G$, we obtain $|\V\Gamma| = |N| = 2$, which goes against $\Gamma$ being cubic.

\smallskip
\noindent\textsc{$\Gamma/N$ has valency $1$.}
In this case $N$ has two orbits on $\V\Gamma$, forming a system of imprimitivity for the action of $G$ on $\V\Gamma$.
Let $G^+$ be the subgroup of $G$ fixing setwise the two $N$-orbits on $\V\Gamma$.
By Frattini's argument, $G^+ = N G_\alpha$. 

If $N$ is an elementary abelian $2$-group, then $G$ is a $2$-group and hence $|N| = 2$,
since $N$ is minimal normal in $G$.
Thus $|\V\Gamma| = 2|N| = 4$, and the proof follows immediately.
Suppose then that $N$ is an elementary abelian $p$-group, with $p$ odd.
In particular, since $G_\alpha$ is a $2$-group, $N$ acts semiregularly on $\V\Gamma$.

Let $\beta \in \Gamma(\alpha)$ and suppose that $\Gamma$ is bipartite,
with bipartition given by the orbits of $N$ on $\V\Gamma$.
Then there exist $n_1, n_2 \in N$ such that
\[
\Gamma(\alpha) = \{\beta, \beta^{n_1}, \beta^{n_2}\}.
\]
Then
\[
\Gamma(\beta) = \{\alpha, \alpha^{n_1^{-1}}, \alpha^{n_2^{-1}}\}.
\]
Let $g \in G_\alpha^{[1]}$.
Then
\[
\beta^{n_i} = (\beta^{n_i})^g = \beta^{n_i g} = \beta^{g n_i^g} = \beta^{n_i^g}.
\]
Since $N$ is normal in $G$, we have $n_i^g \in N$.
As $N$ acts semiregularly on each of its orbits, we deduce that $n_i = n_i^g$, and hence $g$ centralizes $n_i$.
Thus $g$ also centralizes $n_i^{-1}$, and so $g \in G_\beta^{[1]}$.
In particular, $G_\alpha^{[1]} = G_\beta^{[1]}$.  

A straightforward connectedness argument now shows that $G_\alpha^{[1]}$ fixes every vertex of $\Gamma$,
and hence $G_\alpha^{[1]} = 1$.
In particular, $G_\alpha$ acts faithfully on $\Gamma(\alpha)$ and $|G_\alpha|=2$.
This contradiction arises from assuming that $\Gamma$ is bipartite, with parts given by the orbits of $N$ on $\V\Gamma$.

Therefore, the subgraph of $\Gamma$ induced on each $N$-orbit is a disjoint union of cycles,
and $\Gamma$ also has a complete matching between the two $N$-orbits.
Let $\alpha' \in \Gamma(\alpha) \setminus \alpha^N$ be the neighbour of $\alpha$
which does not lie in the same $N$-orbit as $\alpha$.
Thus $G_\alpha=G_{\alpha'}$, and there exist $n,m \in N$ such that
\[
\Gamma(\alpha) = \{\alpha', \alpha^n, \alpha^{n^{-1}}\}
\quad \text{and} \quad
\Gamma(\beta) = \{\alpha, \alpha'^m, \alpha'^{m^{-1}}\}.
\]
Let $x \in G$ with $(\alpha,\alpha')^x = (\alpha',\alpha)$ and set $H = \langle x,n\rangle$.
Observe that $\Gamma(\alpha) \subseteq \alpha^H$.
By connectedness of $\Gamma$, it follows that $H$ is transitive on $\V\Gamma$.

Since $\{\alpha,\alpha^n\}$ is an edge of $\Gamma$, its image under $x$ is also an edge: namely,
\[
\{\alpha,\alpha^n\}^x
=
\{\alpha^x,\alpha^{nx}\}
=
\{\alpha',\alpha'^{\,n^x}\}.
\]
As $N$ acts semiregularly, $n^x \in \{m,m^{-1}\}$.
Replacing $m$ by $m^{-1}$ if necessary, we may assume that $n^x = m$.
Similarly, since $\{\alpha',\alpha'^{\,n^x}\}$ is an edge, so is its image under $x$, given by
\[
\{\alpha',\alpha'^{\,n^x}\}^x
=
\{\alpha,\alpha^{\,n^{x^2}}\}.
\]
Again using semiregularity of $N$, we deduce that $n^{x^2} \in \{n,n^{-1}\}$.
Thus $x^2$ normalises $\langle n\rangle$, and therefore
\[
H = \langle n,x\rangle
=
\langle n,n^x\rangle \rtimes \langle x^2\rangle.
\]

Since $H$ is transitive on $\V\Gamma$, we obtain $N =\langle n,n^x\rangle$.
Thus $N$ is either cyclic of order $p$ or elementary abelian of order $p^2$. 

Let $y \in {\bf C}_{G_\alpha}(N)$. Since $G_\alpha = G_{\alpha'}$, the element $y$ centralizes $\alpha^N \cup \alpha'^N = \V\Gamma$, and hence $y = 1$. Therefore ${\bf C}_{G_\alpha}(N) = 1$, and $G_\alpha$ acts faithfully by conjugation on $N$.

Let $y \in G_\alpha = G_{\alpha'}$. Since $y$ permutes the elements of $\Gamma(\alpha)\setminus\{\alpha'\}$ and $\Gamma(\alpha')\setminus\{\alpha\}$, we deduce that
\[
n^y \in \{n,n^{-1}\}
\quad\text{and}\quad
(n^x)^y \in \{n^x,(n^{-1})^x\}.
\]
As $N = \langle n,n^x\rangle$, it follows that either $|G_\alpha| = 2$ or $G_\alpha$ is elementary abelian of order $4$.
We have already excluded the former case.
In the latter, $N = \langle n,n^x\rangle$ is an elementary abelian $p$-group of order $p^2$, and $G_\alpha = \langle y_1,y_2\rangle$, where
\[
(n,n^x)^{y_1} = (n^{-1},n^x)
\quad\text{and}\quad
(n,n^x)^{y_2} = (n,(n^x)^{-1}).
\]

Let $\gamma = \alpha^{nn^x}$. Suppose $y \in G_\alpha \cap G_\gamma$. Then $y = z^{nn^x}$ for some $z \in G_\alpha$, and hence
\[
z^{-1}y = [z,nn^x] \in N \cap G_\alpha = 1.
\]
Thus $y = z$, and so $z^{nn^x} = z$, that is, $z$ centralizes $nn^x$.
However, the identity is the only element of $G_\alpha$ that centralizes $nn^x$, and hence $G_\alpha \cap G_\gamma = 1$. This shows that $G$ has base size at most $2$, which is a contradiction.

\smallskip
\noindent\textsc{$\Gamma/N$ has valency $2$.}
In this case $\Gamma/N$ is a cycle of length $\ell \ge 3$.
Let $K$ be the kernel of the action of $G$ on the $N$-orbits. Frattini's argument implies $K=NK_\alpha$.
Moreover, $G/K$ is isomorphic either to the cyclic group of order $\ell$ or to the dihedral group of order $2\ell$.

If $N$ is a $2$-group, then \cref{lem:GNc} implies that either
$\Gamma$ is isomorphic to the circular ladder or to the M\"{o}bius ladder,
or $\Gamma$ is isomorphic to $\Spl\C(r,s)$ for some $r$ and $s$.
In the latter case, $\Gamma$ satisfies \cref{thrm:main}.
When $\Gamma = \Cay(\mathbb{Z}_n \times \mathbb{Z}_2, \{(1,0), (-1,0), (0,1)\})$ is the circular ladder,
it is a routine exercise to verify that, when $n \ne 4$,
\[
\Aut(\Gamma)
=
(\mathbb{Z}_n \times \mathbb{Z}_2) \rtimes \langle \iota \rangle,
\]
where $\iota$ is the automorphism of $\mathbb{Z}_n \times \mathbb{Z}_2$ mapping each element to its inverse.
When $n = 4$, the circular ladder is the cube graph.
Similarly, when $\Gamma = \Cay(\mathbb{Z}_{2n}, \{1, -1, n\})$ is the M\"{o}bius ladder,
it is a routine exercise to verify that, for $n \ge 4$,
\[
\Aut(\Gamma)
=
\mathbb{Z}_{2n} \rtimes \langle \iota \rangle,
\]
where, as above, $\iota$ is the automorphism of $\mathbb{Z}_{2n}$ mapping each element to its inverse.
When $n = 3$, the M\"{o}bius ladder is the complete graph on $4$ vertices, and when $n = 2$,
the M\"{o}bius ladder is the complete bipartite graph.
In each case, $\Gamma$ satisfies \cref{thrm:main}.

Suppose that $N$ is an elementary abelian $p$-group with $p$ odd.
In particular, $N \cap G_\alpha =1$ and $N$ acts semiregularly on $\V\Gamma$.
\cref{lem:bigbigbig} and the assumption $|G_\alpha|\ge 4$ imply that the integer $\ell$ is even and $G_\alpha = K_\alpha$ is an elementary abelian $2$-group.

Let $C = \mathbf{C}_K(N)$. As $|C : N|$ is a power of $2$, we have $C = N \times L$ for some $2$-group $L \leqslant G_\alpha$. Since $L$ is characteristic in $C$ and $C \unlhd G$, it follows that $L \unlhd G$.
As $L \leqslant G_\alpha$ and $G_\alpha$ is core-free in $G$, we obtain $L = 1$, and hence $C = N$.
This shows that $G_\alpha$ acts faithfully by conjugation on $N$.

As $N$ is minimal normal in $G$, it can be viewed as an irreducible $\mathbb{F}_p(G/N)$-module,
where $\mathbb{F}_p$ denotes the finite field of order $p$.
Since $K/N \unlhd G/N$, Clifford's theorem implies that $N$ is a completely reducible $\mathbb{F}_p(K/N)$-module.
Therefore, we may write
\[
N = N_1 \times \cdots \times N_\kappa,
\]
where each $N_i$ is an irreducible $\mathbb{F}_p(K/N)$-module.
As $K/N$ is abelian, each $N_i$ has dimension $1$, and $K/N$ acts by scalars $\pm 1$ on $N_i$.
Thus, when we view the elements of $K/N$ acting by conjugation on $N$,
we may interpret them as $\kappa \times \kappa$ diagonal matrices whose diagonal entries are $\pm 1$.

Now, for each $i$, let $x_i \in N_i \setminus \{1\}$ and set
\[
n = x_1 \cdots x_\kappa.
\]
From the description above, if $g \in K$ centralizes $n$, then $g \in \mathbf{C}_K(N) = N$, because no element of $K \setminus N$ can centralize such an element $n$.
Thus no element of $G_\alpha$ centralizes $n$.

We claim that $G_\alpha \cap G_\alpha^n = 1$. Indeed, if $x,y \in G_\alpha$ satisfy $x = y^n$, then
\[
y^{-1} x = y^{-1} n^{-1} y n = [y,n] \in G_\alpha \cap N = N_\alpha = 1.
\]
Thus $x = y$,  and hence $y = y^n$, which implies that $y$ centralizes $n$.
Consequently, $y = 1$ and $G_\alpha\cap G_\alpha^n=1$.
This contradicts the assumption that the base size of $G$ is greater than $2$.

\smallskip
\noindent\textsc{$\Gamma/N$ has valency $3$.}
In particular, $N$ acts semiregularly on $\V\Gamma$ and is the kernel of the action of $G$ on the set of $N$-orbits. Let $\Delta = \Gamma/N$. By hypothesis, $\Delta$ satisfies \cref{thrm:main}. 

Suppose first that the automorphism group of $\Delta$ has base size at most $2$.
Since $G/N \leqslant \Aut(\Delta)$ and $N$ is semiregular, it follows that $G$ also has base size at most $2$. 

Suppose next that the automorphism group of $\Delta$ has base size greater than $2$. Then, by \cref{thrm:main}, the graph $\Delta$ is either one of the exceptional graphs listed in \cref{tab:main}, or $\Delta$ is isomorphic to $\Spl\C(r,s)$ for some integers $r,s$ satisfying $2s \ge r$.

Suppose that $\Delta$ is one of the exceptional graphs in \cref{tab:main}.
A case-by-case analysis shows that,
if $X \leqslant \Aut(\Delta)$ is vertex-transitive and all vertex-stabilisers in $X$ are $2$-groups,
then $X$ has base size at most $2$.
Now, since $G/N \leqslant \Aut(\Delta)$ and $G_\alpha$ is a $2$-group,
we deduce that the stabiliser $(G/N)_{\alpha^N} \cong G_\alpha$ is also a $2$-group.
Therefore, $G/N$ has base size at most $2$ in its action on the vertices of $\Delta$,
and hence $G$ has base size at most $2$ in its action on $\V\Gamma$.
This contradicts our standing assumption, and therefore this case does not arise.

Suppose now that $\Delta$ is isomorphic to $\Spl\C(r,s)$ for some $r$ and $s$.
Let ${\rm m}\Gamma$ be the merge of $\Gamma$.
Since $\Gamma/N$ is isomorphic to $\Spl\C(r,s)$,
it follows from~\cite[Lemma~9 and Theorem~10]{PotocnikSpigaVerret2013} that
\begin{itemize}
\item ${\rm m}\Gamma$ is $4$-valent;
\item $G/N$ acts faithfully on ${\rm m}\Gamma$;
\item ${\rm m}\Gamma/N$ is isomorphic to $\C(r,s)$.
\end{itemize}
Since the base size of $G/N$ in its action on $\V\Gamma$ is greater than $2$, the same holds for its action on $\V{\rm m}\Gamma$. 

If $N$ has even order, then \cref{prop:vCcov} implies that ${\rm m}\Gamma$ is isomorphic to $\C(r',s')$
for some integers $r'$ and $s'$.
Applying again~\cite[Lemma~9 and Theorem~10]{PotocnikSpigaVerret2013},
$\Gamma$ is isomorphic to $\Spl\C(r',s')$.
The necessary arithmetic condition on $r'$ and $s'$ follows from \cref{bloodyhel}. 

Therefore, for the remainder of the argument we may assume that $N$ has odd order.
To complete this last case, we need to borrow ideas both from the proof of \cref{prop:vCcov} and from the cyclic quotient case. Let $E$ be the normal closure of a vertex-stabiliser, that is,
\[ E = \langle G_\beta \mid \beta \in \V\Gamma \rangle ,\]
and let $C = \mathbf{C}_E(N)$.
Observe that, as $N$ is of odd order, reasoning as in (\ref{eq:propE}), $EN/N \leqslant H_{\alpha^N}$.
Hence, either $E$ is not semiregular, or $E$ is the extension of $N$ by a $2$-group.
In the former case, we conclude using \cref{cor:3.4}.
In the latter scenario, $\gcd(|C/N|, |N|) = 1$, and it follows that $C = N \times L$ for some elementary abelian $2$-group $L$. Observe that $C \unlhd G$, as $E$ is characteristic in $G$ and $C \unlhd E$. 
If $L \ne 1$, then we may choose a minimal normal subgroup $N'$ of $G$ contained in $L$.
Since $N'$ is a $2$-group, we may apply the preceding argument with $N$ replaced by $N'$.
Hence, we may assume that $L = 1$, that is, $C = N$.
It follows that $E/N$ (and, therefore, by semiregularity of $N$, $G_\alpha$) acts faithfully by conjugation on $N$.

We are now in a position to apply the same argument used in the case where $\Gamma/N$ has valency $2$,
which we repeat here for completeness.
Since $N$ is minimal normal in $G$, it can be viewed as an irreducible $\mathbb{F}_p(G/N)$-module,
where $\mathbb{F}_p$ denotes the finite field of odd order $p$.
As $E/N \unlhd G/N$, Clifford’s Theorem implies that $N$ is a completely reducible $\mathbb{F}_p(E/N)$-module.
Therefore, we may write
\[
N = N_1 \times \cdots \times N_\kappa,
\]
where each $N_i$ is an irreducible $\mathbb{F}_p(E/N)$-module.
Since $E/N$ is an elementary abelian $2$-group, each $N_i$ has dimension $1$,
and $E/N$ acts by scalars $\pm 1$ on $N_i$.
Thus, when viewing the elements of $E/N$ acting by conjugation on $N$,
we may interpret them as $\kappa \times \kappa$ diagonal matrices whose diagonal entries are $\pm 1$.

Now, for each $i$, let $x_i \in N_i \setminus \{1\}$ and set
\[
n = x_1 \cdots x_\kappa.
\]
From the description above, if $g \in E$ centralizes $n$,
then $g \in \mathbf{C}_E(N) = 1$, because no nontrivial element of $E$ can centralize such an element $n$.
Hence, no nontrivial element of $G_\alpha$ centralizes $n$.

We claim that $G_\alpha \cap G_\alpha^n = 1$. Indeed, if $x,y \in G_\alpha$ satisfy $x = y^n$, then
\[
y^{-1}x = y^{-1}n^{-1}yn = [y,n] \in G_\alpha \cap N = N_\alpha = 1.
\]
Thus $x = y$ and hence $y = y^n$, which implies that $y$ centralizes $n$.
Consequently, $y = 1$, and therefore $G_\alpha \cap G_\alpha^n = 1$.
This contradicts the assumption that the base size of $G$ is greater than $2$.
\end{proof}

\section{Reduction to monolithic groups} \label{sec:nonabelianNonTriv}

We are left with the case where $G \leqslant \Aut(\Gamma)$ has trivial soluble radical.
The aim of this short section is to reduce to the case of a unique minimal normal subgroup.

\begin{proposition}\label{prop:minnonabelnontriv}
Let $\Gamma$ be a connected vertex-transitive cubic graph.
Let $G \leqslant \Aut(\Gamma)$ act transitively on $\V\Gamma$.
Suppose that $G$ has trivial soluble radical
and let $N$ be a minimal normal subgroup such that ${\bf C}_G(N) \ne 1$.
If $\Gamma/N$ has valency $3$, then suppose that \cref{thrm:main} holds for $\Gamma/N$.
Then \cref{thrm:main} holds for $\Gamma$.
\end{proposition}
\begin{proof}
By \cref{prop:arctran}, we can assume $G$ is not arc-transitive.
By Proposition~\ref{prop:minabel},
we may also assume that every nontrivial normal subgroup of $G$ is nonabelian.
If $\alpha \in \V\Gamma$, then, by \cref{lem:Galpha},
$G_\alpha$ is a $2$-group of nilpotency class at most $2$ and exponent at most $4$.
Moreover, we may suppose that $|G_\alpha| \ge 4$, and we divide the proof according to the valency of $\Gamma/N$.

\smallskip
\noindent\textsc{$\Gamma/N$ has valency $0$.}
Then $N$ is transitive on $\V\Gamma$, and hence $G = G_\alpha N$.
Since $G/N$ is a $2$-group and ${\bf C}_G(N) \ne 1$,
it follows that $G$ has a minimal normal subgroup that is a $2$-group, which is impossible.

\smallskip
\noindent\textsc{$\Gamma/N$ has valency $1$.}
In this case, $N$ has two orbits on $\V\Gamma$, and $G$ admits a normal subgroup $G^+$
fixing setwise the two $N$-orbits, with $|G : G^+| = 2$.
By Frattini's argument, $G^+ = G_\alpha N$, and hence $G/N$ is a $2$-group.
As before, this implies that ${\bf C}_G(N)$ is a nontrivial $2$-group, which is impossible.

\smallskip
\noindent\textsc{$\Gamma/N$ has valency $2$.}
Let $M \leqslant {\bf C}_G(N)$ be a normal subgroup of $G$ distinct from $N$,
and let $K$ be the kernel of the action of $G$ on the set of $N$-orbits.
Since $\Aut(\Gamma/N)$ is dihedral, the quotient $G/K$ is soluble.
Therefore, the nonabelian minimal normal subgroup $M$ must be contained in $K$. In particular,
\[
\alpha^M \subseteq \alpha^K = \alpha^N.
\]
Let $m \in M \setminus \{1\}$ be an element of odd order.
Since $\alpha^m \in \alpha^N$, there exists $n \in N$ such that $\alpha^m = \alpha^n$, that is, $mn^{-1} \in G_\alpha$.
As $mn^{-1} \in M \times N$, the element $mn^{-1}$ has order divisible by an odd prime.
This contradicts the fact that $G_\alpha$ is a $2$-group.

\smallskip
\noindent\textsc{$\Gamma/N$ has valency $3$.}
Let $M \leqslant {\bf C}_G(N)$ be a normal subgroup of $G$ distinct from $N$,
and observe that we may assume that $\Gamma/M$ has valency $3$.
Otherwise, an entirely analogous argument with the roles of $N$ and $M$ swapped would make us conclude.
Moreover, we may suppose that both $\Gamma/N$ and $\Gamma/M$ are graphs appearing in the conclusion of\cref{thrm:main}.
Observe that both $G/N$ and $G/M$ act faithfully on the corresponding quotient graphs.

In particular, the group $NM/N \cong M$ is a section of $\Aut(\Gamma/N)$
and $NM/M \cong N$ is a section of $\Aut(\Gamma/M)$. 
Assume, for a contradiction, that $\Gamma/N \cong \Spl\C(r,s)$ for some $r$ and $s$.
Then $M $ embeds into $\Aut(\Spl\C(r,s))$, which is impossible because
$\Aut(\Spl\C(r,s))$ is soluble.
Hence both $\Gamma/N$ and $\Gamma/M$ must be among the exceptional graphs in \cref{tab:main}.

The graphs in rows~1,~2,~3, and~6 cannot arise, since their automorphism groups are soluble.
The graph in row~9 can also be excluded: one checks that every transitive subgroup of its automorphism group
has a unique minimal normal subgroup of order $3$.
Finally, the graphs in rows~4,~5,~7, and~8 cannot arise,
because no transitive subgroup of their automorphism groups has vertex-stabilisers
that are $2$-groups of order at least $4$.    
\end{proof}

\section{Permutation $2$-groups and colourings} \label{sec:Col}

In this section we deviate from graphs and deal with abstract permutation groups.
For $G \leqslant \Sym(\Omega)$ and $j \geq 1$, a {\em colouring} is a function $\Omega \to \{1,\ldots,j\}$
(here $\{1,\ldots,j\}$ is the set of {\em colours}).
A colouring is {\em asymmetric} if the identity is the only element of $G$ fixing the colouring.
In 1983, Gluck \cite{Glu83} showed that every permutation group of odd order
admits an asymmetric colouring with at most $2$ colours.
Note that this is the same as asking for a subset $X \subseteq \Omega$ such that
the identity is the only element of $G$ that preserves $X$.
By the dihedral group $D_8$ in its action of degree $4$, the odd-order assumption in Gluck's theorem is critical.
We observe that this is the only exception in the range of transitive $2$-groups of class at most $2$.

\begin{lemma}\label{lem:partitions2} 
    Let $P \leqslant \Sym(\Omega)$ be a transitive permutation $2$-group of nilpotency class at most $2$.
    Suppose that $P \neq D_8$.
    Then there exists $X \subseteq \Omega$ such that the identity is the only element of $P$ that preserves $X$.
    Moreover, when $|\Omega|>8$, we can choose $X$ with $|X|<|\Omega|/2$.
\end{lemma}
\begin{proof}
    We have verified by computer the veracity of the statement when $|\Omega| \leq 16$.
    We argue by induction on $|\Omega|$, and thus we can assume that $|\Omega| > 16$.
    Let $Z$ be a central subgroup of $P$ of order $2$, and let
    \[ K = \bigcap_{\alpha \in \Omega} ZP_\alpha\]
    be the kernel of the action of $P$ on the set $\Sigma$ of $Z$-orbits.
    Since $|\Sigma|=|\Omega|/2>8$, by induction there exists $X' \subseteq \Sigma$, with $|X'| < |\Sigma|/2$,
    whose setwise stabiliser in $P$ lies in $K$.
    We build $X \subseteq \Omega$ by choosing, from each $Z$-orbit in $\Sigma \setminus X'$, 
    exactly one point, and not including any points lying in the $Z$-orbits belonging to $X'$.
    Note that
    \[|X| < |\Sigma| = |\Omega|/2 .\]
    Let $x \in P$ be a permutation that stabilises $X$.
    We aim to prove that $x$ is trivial.
    By construction, $x$ fixes $X'$ setwise, and hence lies in $K$.
    Moreover, $x$ fixes at least $|\Omega \setminus X| > |\Omega|/2$ points.
    Choose a point $\alpha \in \Omega$ such that $x \in P_\alpha$. By (\ref{eq:fpr}),
    \begin{equation} \label{eq:color}
     \frac{1}{2} < \mathrm{fpr}_\Omega(x) = \frac{|x^P \cap P_\alpha|}{|x^P|} .
    \end{equation}
    We claim that $|x^P|=1$.
    For every $g \in P$ such that $x^g \in P_\alpha$, $[x,g] \in P_\alpha \cap [P,P]$.
    Since $[P,P] \subseteq {\bf Z}(P)$ by hypothesis, $P_\alpha \cap [P,P]$ is a normal subgroup of $P$ contained in $P_\alpha$.   	As $P_\alpha$ is core-free, $P_\alpha \cap [P,P]$ is trivial.
    Hence, $x^P \cap P_\alpha = \{x\}$, and, by (\ref{eq:color}), $|x^P|=1$, which proves the claim.
    In particular, $x \in P_\alpha \cap {\bf Z}(P) =1$, as desired.
\end{proof}

We now show that, when three colours are allowed, every finite $2$-group admits an asymmetric colouring.
We actually prove a stronger statement that is suitable for an inductive argument.
Two colourings are {\em inequivalent} if they lie in distinct orbits under the action of $P$.

\begin{lemma} \label{lem:partitions}
Let $|\Omega| \geq 2$ and let $P \leqslant \Sym(\Omega)$ be a finite $2$-group.
Then there exists at least $3$ inequivalent asymmetric colourings with at most $3$ colours.
\end{lemma}
\begin{proof}
We follow the strategy described in \cite[Section~2]{Sabatini2026}.
We can assume that $P$ is a Sylow $2$-subgroup of $\Sym(\Omega)$, and work by induction on $|\Omega|$.
Moreover, if $P$ is not transitive, then $\Omega$ is the disjoint union of its $P$-orbits.
By induction we may colour each orbit, and combining these colourings
yields asymmetric colourings of $\Omega$ with at most $3$ colours.

If $P_k = C_2 \wr_2 \ldots \wr_2 C_2$ is the $k$-fold wreath product,
then we may assume $P = P_k$ and $|\Omega|=2^k$ for some $k \geq 1$ and .
The claim holds for $k=1$.
For each $k \geq 1$, we use the $3$ inequivalent colourings of $C_2$ as colours for $P_k$.
This generates a colouring of $P_{k+1}$ from each colouring of $P_k$.
Since inequivalent colourings of $P_k$ provide inequivalent colourings of $P_{k+1}$ \cite[Lemma~2.9]{Sabatini2026},
the proof follows.
\end{proof}

\section{Intersections of $2$-subgroups in almost simple groups} \label{sec:nonabelian}

In the proof of \cref{thrm:main},
we are left with $G \leqslant \Aut(\Gamma)$ being a monolithic group with nonabelian socle.
The following fundamental result is purely group theoretical and will be proven in the next section.
    
    \begin{proposition} \label{prop:minnonabel}
    Let $G$ be a monolithic group with nonabelian socle $M$,
    and let $P$ be a $2$-subgroup of $G$ having nilpotency class at most $2$ and exponent at most $4$.
    Then there exists $m \in M$ such that $P \cap P^m = 1$.  
\end{proposition}
   
    Our goal for now is to obtain a strong version of \cref{prop:minnonabel} in the case where $G$ is almost simple,
    i.e. when $M=T$ is a nonabelian simple group.
    For an almost simple group $G$ with socle $T$ and a $2$-subgroup $P$ of $G$,
    we consider the following property:
    
    \begin{equation} \label{eqStar}
    \begin{gathered}
    \text{there exist } t_1,t_2,t_3 \in T \text{ such that } P \cap P^{t_i} =1 \text{ for each } i = 1,2,3,    \\
    \text{and the double cosets } Pt_iP \text{ are pairwise disjoint.} 
    \end{gathered}
    \tag{$\star$}
    \end{equation} 

\begin{remark}
The orbits of \(P\) on the coset space \(G/P\) are naturally parametrised by the double cosets \(PgP\).
A straightforward computation shows that
the stabiliser in \(P\) of a point in the orbit corresponding to \(PgP\) is conjugate to \(P \cap P^g\).
In particular, this orbit is regular if and only if \(P \cap P^{g}=1\).
Thus (\ref{eqStar}) is equivalent to the statement that \(P\) has at least three regular orbits
of type $PtP$ ($t \in T$) in its action on \(G/P\).
\end{remark}
    
    The proofs in this section differ from the rest of the paper,
    as they rely on deep properties of the finite simple groups.
While proving a statement like \cref{prop:minnonabel},
it is tempting to embed $P$ into a Sylow $2$-subgroup of $G$, say $Q$,
and to seek for $t \in T$ such that $Q \cap Q^t =1$.
Unfortunately, this approach fails badly in general, as there are examples where $Q \cap Q^g \neq 1$ for all $g \in G$.
In fact, all the exceptions have been recently described by Burness and Huang.

\begin{theorem}[\cite{BurHua2026}, Theorem~A] \label{thBurHua}
Let $G$ be an almost simple group and let $Q$ be a Sylow $2$-subgroup of $G$.
Then either there exists $g \in G$ such that $Q \cap Q^g =1$,
or $G$ is recorded in \cref{tab:BH}.
\end{theorem}

\begin{table}[ht]
\begin{center}
\begin{tabular}{||c||c|c||}
\hline\hline
&$G$ & conditions \\
 \hline\hline
1& $\PSL_2(9) . 2^2$  &  \\
2& $\PSL_2(q) . 2$  &  $q \geq 7$ is a Mersenne prime \\
 3& $\PSL_3(4) . X$  & $X \in \{C_2,C_2^2\}$ \\
 4& $\PSL_n(2) . 2$  & $n \geq 4$ \\
  5& $\mathrm{O}_{2n}^+(2)$  & $n \geq 4$  \\
  6& $F_4(2) .2$  &  \\
7& $E_6(2).2$  &  \\
\hline\hline
\end{tabular}
\medskip
\caption{The almost simple groups $G$ such that $Q \cap Q^g \neq 1$ for all $g \in G$.}
 \label{tab:BH}
\end{center}
\end{table}

More information about the groups in \cref{tab:BH} can be found in \cite[Corollary~B and Remark~4]{BurHua2026}.

\begin{lemma} \label{lem:table}
Let $G$ be one of the almost simple groups in \cref{tab:BH}.
Let $T$ be the socle of $G$, and let $P$ be a $2$-subgroup of $G$
having nilpotency class at most $2$ and exponent at most $4$.
Then $(G,P)$ satisfies \textup{(\ref{eqStar})}.
\end{lemma}
\begin{proof}
The rows 1 and 3 in \cref{tab:BH} can be handled computationally.
In the following, we will always apply \cite[Corollary~B and Remark~4]{BurHua2026} to the almost simple group $TP$.
Let $Q$ be a Sylow $2$-subgroup of $G$ containing $P$.

\smallskip
\noindent
\textsc{Row 2.} 
In this case $Q$ is dihedral of order $2(q+1)$,
and the largest subgroup of $Q$ having nilpotency class at most $2$ and exponent at most $4$ has order $8$.
If $q=7$ then we can proceed computationally, otherwise
\[
|Q:P|\ge \frac{q+1}{4} >3 .
\]

We now apply \cite[Corollary~B]{BurHua2026} with $A = P$ and $B = Q$.
Since $A$ and $B$ are not both Sylow $2$-subgroups of $G$, we deduce that $P \cap Q^t = 1$ for some $t \in T$. 
Let $\Delta$ be a transversal for $P$ in $Q$.
For every $x \in \Delta$, we have
\[
P \cap Q^{xt} = P \cap Q^t = 1 .
\]
Let $x,y \in \Delta$ and suppose that $PxtP = PytP$.
Then $xt = a y t b$ for some $a,b \in P$, and rearranging we obtain
\[
b^{-1} = t^{-1} x^{-1} a y t \in P \cap Q^t = 1 .
\]
It follows that $x = ay$, and since $x$ and $y$ lie in the transversal $\Delta$, this implies $x = y$.
Thus the double cosets $PxtP$, as $x$ ranges over $\Delta$, are pairwise disjoint.
Since $|\Delta| > 3$, the result follows in this case.

\smallskip
\noindent
\textsc{Row 4.} 
For $n \leq 6$ the statement has been verified computationally. Hence, we assume $n\ge 8$.
Since $\SL_8(2).2$ is a subgroup of $\SL_n(2).2$, a direct computation for $\SL_8(2).2$ shows that, in any case, $Q$ has exponent at least $16$.
Since $P$ has exponent at most $4$, it follows that $|Q:P|\ge 4$; otherwise $Q$ would have exponent at most $8$, a contradiction.

Applying \cite[Corollary~B]{BurHua2026} with $A=P$ and $B=Q$, we obtain either $P\cap Q^t=1$ for some $t\in T$, or $K\leqslant P\leqslant Q$, where $K$ is defined in \cite[Remark~4]{BurHua2026}. From \cite[Remark~4]{BurHua2026} we have
\[|Q:K|=2<|Q:P| .\]
Therefore, the second possibility cannot occur.
Thus $P\cap Q^t=1$ for some $t\in T$, and as above the transversal argument yields the desired elements.

\smallskip
\noindent
\textsc{Row 5.}
When $n=4$ the statement has been verified computationally, so assume $n\ge 5$.
As $\mathrm{O}_{10}^+(2)$ is a subgroup of $\mathrm{O}_{2n}^+(2)$, $Q$ has exponent at least $16$, and hence $|Q:P|\ge 4$ by the same argument as in the previous case.

Applying \cite[Corollary~B]{BurHua2026} again, either $P\cap Q^t=1$ for some $t\in T$, or $K\leqslant P\leqslant Q$.
Here $K$ is the unipotent radical of the stabiliser in $G=\mathrm{O}_{2n}^+(2)$ of a totally singular subspace of dimension $n-1$, and the Levi factor of this stabiliser has a quotient isomorphic to $\SL_{n-1}(2)$. 
Since a Sylow $2$-subgroup of $\SL_4(2)$ has nilpotency class $3$ and $n-1\ge 4$,
the inclusion $K \leqslant P$ is impossible.
Hence $P\cap Q^t=1$ for some $t\in T$, and the conclusion follows as before.

\smallskip
\noindent
\textsc{Row 6.}
An explicit computation shows that every subgroup of $Q$ of index at most $8$ has nilpotency class at least $3$
(actually the least nilpotency class of a subgroup of $Q$ of index at most $8$ is $4$). 
Therefore $|Q:P|\ge 16$. (This bound will also be used in the next case.)

Applying \cite[Corollary~B]{BurHua2026}, either $P\cap Q^t=1$ for some $t\in T$, or $K\leqslant P\leqslant Q$.
From \cite[Remark~4]{BurHua2026} we have $|Q:K|=2$, so the second possibility cannot occur.
Thus $P\cap Q^t=1$, and since $|Q:P|\ge 4$, the argument with a transversal yields the required elements.

\smallskip
\noindent
\textsc{Row 7.}
Since $F_4(2).2$ is a subgroup of $E_6(2).2$ (see \cite{Atlas1985}), we have $|Q:P|\ge 16$.
Applying \cite[Corollary~B]{BurHua2026} as above,
either $P\cap Q^t=1$ for some $t\in T$, or $K\leqslant P\leqslant Q$.
From \cite[Remark~4]{BurHua2026} we have $|Q:K|=8$, so the second possibility is excluded.
Hence $P\cap Q^t=1$, and we conclude as in the previous cases.
\end{proof}

We are ready to prove our strong result for almost simple groups.

\begin{lemma}[Almost simple case] \label{lem:almostsimple} 
Let $G$ be an almost simple group with socle $T$,
and let $P$ be a $2$-subgroup of $G$ having nilpotency class at most $2$ and exponent at most $4$.
Then either $(G,P)$ satisfies \textup{(\ref{eqStar})}, or one of the following holds:
\begin{enumerate}[label=(\roman*)]
\item\label{uno} $G = \Sym(5)$ and $P$ is a Sylow $2$-subgroup of $G$;
\item\label{due} $\PSL_2(7)\leqslant G \leqslant \PGL_2(7)$ and $P$ is a Sylow $2$-subgroup of $\PSL_2(7)$;
\item\label{tre} $\Sym(6)\leqslant G \leqslant \Aut(\Alt(6))$ and $P$ is a Sylow $2$-subgroup of $\Sym(6)$;
\item\label{quattro} $G=\Sym(8)$ and, up to conjugation,
$$P=\langle(1\,7)(3\, 5),
(1\, 2\, 7\, 6)(3\, 8)(4\, 5),
(1\, 2)(3\, 5)(4\, 8)(6\, 7),
(1\, 2\, 7\, 6)(3\, 4\, 5\, 8) \rangle$$
has order $64$.
\end{enumerate}
In all cases, \cref{prop:minnonabel} holds for $(G,P)$.
\end{lemma}
\begin{proof}
We immediately observe that, in the cases \ref{uno}-\ref{due}-\ref{tre}-\ref{quattro},
it can be checked with a computer the existence of $t \in T$ such that $P \cap P^t =1$.
Also, by \cref{lem:table} we can assume that $G$ does not appear in \cref{tab:BH},
so by \cref{thBurHua} we can assume that a Sylow $2$-subgroup $Q$ of $G$ has a regular orbit on $G/Q$.
When $T\in \{\Alt(5),\Alt(6),\Alt(8),\PSL_2(7)\}$, the proof follows with a computer calculation.
We will prove that (\ref{eqStar}) holds in all the remaining cases.

Let $\Delta = G/P$ and let $Q$ be a Sylow $2$-subgroup of $G$ containing $P$.
If $Q \neq P$, then \cref{lem:inflate} implies that $P$ has at least $|Q:P|^2\ge 4$ regular orbits on $\Delta$.
Therefore, for the rest of the proof we can assume $P=Q$.
In particular, $G$, and so $T$, has a Sylow $2$-subgroup of nilpotency class at most $2$.

\smallskip
\noindent\textsc{$T$ has an abelian Sylow $2$-subgroup.}
By a classical result of Walter \cite{MR249504}, the group $T$ is isomorphic to one of the following:
\begin{enumerate}[$(A)$]
\item\label{A.1} $\SL_2(2^f)$, with $f \ge 2$;
\item\label{B.1} $\PSL_2(q)$, with $q \equiv 3,5 \pmod{8}$;
\item\label{C.1} ${}^2G_2(3^f)$, with $f$ odd and $f \ge 3$;
\item\label{D.1} $J_1$.
\end{enumerate}
We treat each case separately.

\noindent\emph{Case \ref{A.1}.}
Let $V$ be the $2$-dimensional vector space over the finite field $\mathbb{F}_{2^f}$,
and let $\Sigma = V \setminus \{0\}$. Consider the natural action of $G$ on $\Sigma$.
Since $|P|$ is a power of $2$ and $|\Sigma|$ is odd, the group $P$ fixes some element $\sigma \in \Sigma$.
Hence $\Sigma$ is a system of imprimitivity for the action of $G$ on $\Delta$.

We aim to determine a lower bound for the number of regular orbits of $G_\sigma$ on $\Sigma$.
Let $e_1,e_2$ be a basis of $V$, let $q=2^f$,
and let $\phi \colon \mathbb{F}_q \to \mathbb{F}_q$ be the field automorphism defined by $x^\phi = x^{2^{f_o}}$.
We write $f=f_2 f_o$, where $f_2$ is a power of $2$ and $f_o$ is odd.
Note that, because $G=TP$, $\phi$ has order $f_2$ and $G \leqslant \SL_2(2^f)\rtimes\langle\phi\rangle$.
The stabiliser of $e_1$ in $\SL_2(2^f)\rtimes\langle\phi\rangle$ is
\[
X=\left\{
\begin{pmatrix}
1 & 0\\
x & 1
\end{pmatrix}
\;\middle|\;
x \in \mathbb{F}_q
\right\} \rtimes \langle \phi \rangle.
\]
It follows that $P\leqslant X$.
We define
\[\phi_0=
\begin{cases}
    \phi^{f_2/2} & \text{if } f_2 \ge 2, \\
    1 & \text{if } f_2 = 1,
\end{cases}\]
and
\[
V_0 = \{ \mu_1 e_1 + \mu_2 e_2 \mid \mu_1 \in \mathbb{F}_q,\; \mu_2 \in \mathbb{F}_{q}\setminus \Fix_{\mathbb{F}_q}(\phi_0) \}.
\]
A direct verification shows that the identity element of $X$ is the only element that fixes any vector of $V_0$.
Hence, every element of $V_0$ lies in a regular $X$-orbit.
Therefore, $X$ has at least
\[
\frac{|V_0|}{|Q|}
\ge \frac{q(q-\sqrt{q})}{q f_2}
= \frac{q-\sqrt{q}}{f_2}
\]
regular orbits on $V$.
A straightforward computation shows that $(q-\sqrt{q})/f_2 \ge 3$ for all $f \ge 3$,
thus completing the analysis of this scenario.

When $f=2$, $T \cong \Alt(5)$, which we already dealt with computationally:
the exception to (\ref{eqStar}) arising is precisely the one described in~\ref{uno}.

\noindent\emph{Case \ref{B.1}.}
Write $q=p^f$, where $p$ is prime. If $f$ were even, then $p^f \equiv 1,7 \pmod{8}$, a contradiction.
Hence $f$ is odd. Since $G=TP$, either $G=\PSL_2(q)$ or $G=\PGL_2(q)$.

A Sylow $2$-subgroup of $H$ is either elementary abelian of order $4$ or dihedral of order $8$,
and hence $P$ contains at most $5$ involutions.
Assume $q>11$.
Since the minimal degree of a faithful permutation representation of $G$ is $q+1$, from (\ref{eq:fpr}), we obtain
\[
\sum_{\substack{g \in P \\ {\bf o}(g)=2}} \mathrm{fpr}_\Delta(g)\le \sum_{\substack{g \in P \\ {\bf o}(g)=2}} \frac{5}{|G:{\bf C}_H(h)|}\le
  \sum_{\substack{g \in P \\ {\bf o}(g)=2}} \frac{5}{q+1}
\le \frac{5^2}{q+1}.
\]
By direct computation, $5^2/(q+1) \le 1 - 3|P|/|\Delta|$ holds for all $q \notin \{13,19\}$.
Hence, when $q>19$, we can apply \cref{lem:threetimess}.

For $q \le 19$, the statement has been verified computationally, and the only exception is given in \ref{due}.

\noindent\emph{Case \ref{C.1}.}
From \cite[Table~5]{Atlas1985}, the group $\mathrm{Out}(T)$ has odd order, and hence $G=TP=T$.
Thus, $P$ is a Sylow $2$-subgroup of $T$.

By \cite{MR138680}, the group $P$ is elementary abelian of order $8$,
and $T$ has a unique conjugacy class of involutions $g$ with ${\bf C}_T(g) \cong 2 \times \PSL_2(3^f)$.
For $q=3^f$, since $|T|=(q^3+1)q^3(q-1)$ and $|{\bf C}_T(g)|=q(q-1)$, we obtain
\[
\sum_{\substack{g \in P \\ {\bf o}(g)=2}} \mathrm{fpr}_\Delta(g)
= \frac{7^2}{(q^2-q+1)q^2} \le \frac{7}{2 \cdot 3^3}.
\]
Therefore, we can apply \cref{lem:threetimess}.

\noindent\emph{Case \ref{D.1}.}
From \cite[p.~36]{Atlas1985}, $G=J_1$.
Hence, $P$ is a Sylow $2$-subgroup of $T$.
A direct computation shows that $P$ has $2\,835$ regular orbits on $\Delta$, and the conclusion follows.

\smallskip
\noindent\textsc{$T$ has a Sylow $2$-subgroup of class $2$.}
By a result of Gilman and Gorenstein \cite{MR379662}, the group $T$ is isomorphic to one of the following:
\begin{enumerate}[$(A)$]
    \item\label{A.2} $\Alt(7)$;
    \item\label{B.2} $\PSL_2(q)$, with $q \equiv \pm 7 \pmod{16}$;
    \item\label{C.2} ${}^2B_2(2^f)$, with $f$ odd and $f \ge 3$;
    \item\label{D.2} $\PSL_3(2^f)$, $\mathrm{PSU}_3(2^f)$, $\mathrm{PSp}_4(2^f)$, with $f \ge 2$.
\end{enumerate}

\noindent\emph{Case \ref{A.2}.}
Here, the statement has been verified computationally.

\noindent\emph{Case \ref{B.2}.}
We argue as in the case $q\equiv 3,5\pmod 8$. If $f\equiv 0\pmod 4$, then $p^f \equiv 1 \pmod{16}$, a contradiction.
Hence $f\equiv 1,2,3\pmod 4$, and so a Sylow $2$-subgroup of $G$ has order at most $32$,
and hence $P$ contains at most $31$ involutions.
Assume $q>11$.
Since the minimal degree of a faithful permutation representation of $H$ is $q+1$, from (\ref{eq:fpr}), we obtain
\[
\sum_{\substack{g \in P \\ {\bf o}(g)=2}} \mathrm{fpr}_\Delta(g)
\le \frac{31^2}{q+1}.
\]
For all $q >960$, $31^2/(q+1) \le 1 - 3|P|/|\Delta|$ holds, and thus we can apply \cref{lem:threetimess}.
For $q \le 960$, the statement has been verified computationally,
and the only exceptions are again the cases in~\ref{due} and~\ref{tre}.

\noindent\emph{Case \ref{C.2}.}
We follow \cite{MR136646} for basic properties of $T$.
For $q=2^f$, we have $|T|=(q^2+1)q^2(q-1)$.
Since $\mathrm{Out}(T)$ has odd order, $G=T$ and $P$ is a Sylow $2$-subgroup of $H$. 

Observe that $P$ is trivially intersecting in $T$, that is, $P \cap P^h = 1$ for every $g \in G \setminus {\bf N}_G(P)$.
Recall also that $|{\bf N}_G(P):P|=q-1$.
It follows that
\[
\bigcup_{g \in P,\, g \ne 1} \Fix_\Delta(g)
= \Fix_\Delta(P).
\]
Hence, this set has cardinality $q-1$.
Using the same reasoning as in \cref{lem:threetimess}, we deduce that $P$ has
\[
\frac{|\Delta|-(q-1)}{q^2}
= \frac{(q^2+1)(q-1)-q}{q^2}
= q-1
\]
regular orbits on $\Delta$. Since $q=2^f \ge 8$, the result follows immediately.

\noindent\emph{Case \ref{D.2}.}
Recall that in these cases we are analysing $P=Q$ is a Sylow $2$-subgroup of $G=PT$.
Using the fact that $P$ has nilpotency class $2$, in fact $G=T$.
Now, let $P^-$ be the opposite unipotent subgroup of $T$.
Then clearly $P\cap P^-=1$. Observe that ${\bf N}_T(P)$ is a Borel subgroup of $T$ and 
\[|{\bf N}_T(P):P|
=
\begin{cases}
(q-1)^2/\gcd(3,q-1)&\textrm{when }T=\PSL_3(q),\\
(q^2-1)/\gcd(3,q+1)&\textrm{when }T=\mathrm{PSU}_3(q),\\
(q-1)^2&\textrm{when }T=\mathrm{PSp}_4(q).
\end{cases}
\]
Since $P$ has at least one regular orbit on $\Delta$, $P$ has at least $|{\bf N}_T(P):P|$ regular orbits on $\Delta$.
Using the fact that $f\ge 2$, we have $|{\bf N}_G(P):P|\ge 3$ and hence $P$ has at least three regular orbits on $\Delta$.

This completes the analysis in the case where $P$ is a Sylow $2$-subgroup of $G$,
and it also completes the proof of the lemma.
\end{proof}

\begin{remark} \label{rem:prob2}
This section is the second major step in the proof of \cref{thrm:main} that we are not able to adapt
to obtain an analogue for connected $4$-valent vertex- and edge-transitive graphs.
The difficulty lies in the fact that, although a $4$-valent analogue of \cref{lem:Galpha} can be obtained,
the nilpotency class of a vertex-stabiliser is bounded only by $5$.
This makes the arguments substantially more complicated.
\end{remark}

\section{Monolithic groups} \label{sec:proof} 

We are ready to prove our main result for monolithic groups,
which is the last ingredient in the proof of \cref{thrm:main}.
One of the key ideas is that the three distinct cosets in (\ref{eqStar}) are used to apply \cref{lem:partitions}.

\begin{proof}[Proof of \cref{prop:minnonabel}]
Let $G$ be a monolithic group with nonabelian socle $M=T^\ell$,
    and let $P$ be a $2$-subgroup of $G$ having nilpotency class at most $2$ and exponent at most $4$.
    We have to show the existence of $m \in M$ such that $P \cap P^m = 1$.  

    Since ${\bf C}_G(M) = 1$, we may view $G$ as a subgroup of
    \[ \Aut(M)=\Aut(T)\, \wr_\ell \,\Sym(\ell) . \]
    We write each element of $G$ in the form $(\varphi_1,\ldots,\varphi_\ell)\sigma$,
    where $\varphi_1,\ldots,\varphi_\ell \in \Aut(T)$ and $\sigma \in \Sym(\ell)$.
    For each $i$,
    let $\pi_i \colon {\bf N}_G(T_i) \to \Aut(T)$ be the projection on the $i$-th direct factor of the base group of $\Aut(M)$.
    
    Consider the action that $P$ induces by conjugation on the direct factors $\{T_1,\dots, T_\ell\}$,
    and denote by $Y_1$, $\dots$, $Y_r$ the orbits of this action.
    The subgroup of $M$ generated by the factors of an orbit $Y_k$ will be denoted by $M_k$,
    and let $\rho_k \colon P \to \Aut(M_k)$ be the restriction map.
    For each $k$, from \cite[Embedding theorem]{MR3791829},
    we may suppose that, for every $T_i$ and $T_j$ in the same orbit $Y_k$,
    \[ \pi_{i}({\bf N}_{P}(T_{i})) = \pi_{j}({\bf N}_P(T_{j})) .\]
    Thus, we can denote by $Q_k$ this common image, and we write $H_k=Q_kT$.
    Note that $H_k$ is almost simple and $Q_k$ is a $2$-subgroup having nilpotency class at most $2$ and exponent at most $4$.
    Moreover, $\rho_k(P) \leqslant Q_k \wr_{Y_k} \Sym(Y_k)$.
   
    We will prove that, for every $k\in\{1,\ldots,r\}$, there exists $m_k \in M_k$ such that
    \begin{equation} \label{eq:final}
    \rho_k(P) \cap \rho_k(P)^{m_k} =1 .
    \end{equation}
    Let $m=m_1 \cdots m_r \in M$, so that for each $k$ we have $\rho_k(P^m) = \rho_k(P)^{m_k}$.
    So if $x \in P \cap P^m$, then $\rho_k(x) \in \rho_k(P) \cap \rho_k(P)^{m_k} =1$,
    and since $k$ is arbitrary we conclude that $x=1$ as desired.

     For each fixed $k\in\{1,\ldots,r\}$, we now divide the proof of (\ref{eq:final}) in two cases.
    
    \medskip \noindent {\sc $(H_k,Q_k)$ satisfies (\ref{eqStar}).}
    There exist $t_1,t_2,t_3 \in T$ such that $Q_k \cap Q_k^{t_i} =1$ for each $i$
    and the double cosets $Q_k t_i Q_k$ are pairwise distinct.
    By \cref{lem:partitions} applied to the group $P$ acting on $Y_k$,
    there exists a partition of the domain, say $Y_k = X_1 \cup X_2 \cup X_3$,
    such that the only elements of $P$ fixing setwise $X_1$, $X_2$, and $X_3$
    are the elements in the kernel of the action of $P$ on $Y_k$.

   Define $m_k \in M_k$ by
   $$ m_k = (u_i)_{i \in Y_k} , \qquad u_i = t_j \text{ if } T_i \in X_j . $$
   Let $x \in \rho_k(P) \cap \rho_k(P)^{m_k}$, and write $x = (q_i)_{i \in Y_k} \sigma$,
   with $q_i \in Q_k$ and $\sigma \in \Sym(Y_k)$.
The \(i\)-th coordinate of \(x^{m_k^{-1}} \in \rho_k(P)\) is $u_i q_i u_{i^\sigma}^{-1} \in Q_k$, so
\[
Q_ku_iQ_k=Q_ku_{i^\sigma}Q_k.
\]
This implies that \(\sigma\) fixes each \(X_j\) setwise, i.e. $x=(q_i)_{i\in Y_k}$ belongs to the base group.
Now let \(T_i\in X_j\). The \(i\)-th coordinate of \(x^{m_k^{-1}} \in \rho_k(P)\) is $t_j q_i t_j^{-1} \in Q_k$, so
\[
q_i\in Q_k\cap Q_k^{t_j} =1 .
\]
Repetead for each $i$, this argument gives $x=1$. Because $x$ is arbitrary, (\ref{eq:final}) follows.
  
\medskip \noindent {\sc $(H_k,Q_k)$ does not satisfy (\ref{eqStar}).}
By \cref{lem:almostsimple}, $(H_k,Q_k)$ is one of the pairs listed in \ref{uno}, \ref{due}, \ref{tre}, or \ref{quattro}.
Moreover, there exists $t \in T$ such that $Q_k \cap Q_k^t =1$.
Moreover, if $|Y_k| \le 8$, then we can consider $\Aut(T) \wr_{Y_k} \Sym(Y_k)$ in each case and
verify computationally that (\ref{eq:final}) holds for some $m_k \in M_k$.

If $|Y_k| >8$, then by \cref{lem:partitions2} there exists $X\subseteq Y_k$ with $2|X|<|Y_k|$
such that the only elements of $P$ fixing setwise $X$ are the elements in the kernel of this action.
For every \(T_i\in Y_k\), define
\[
u_i=
\begin{cases}
1, & \text{if } T_i\in X,\\[2mm]
t, & \text{if } T_i\in Y_k\setminus X,
\end{cases}
\]
and set
$$ m_k=(u_i)_{i\in Y_k}\in M_k. $$

Let
\(
x\in \rho_k(P)\cap \rho_k(P)^{m_k} 
\),
and write $x=(q_i)_{i\in Y_k}\sigma$,
where \(q_i\in Q_k\) for every \(i\in Y_k\), and \(\sigma\) is induced by the action of \(P\) on \(Y_k\).
As before, the \(i\)-th coordinate of \(x^{m_k^{-1}} \in \rho_k(P)\) is $u_i q_i u_{i^\sigma}^{-1}\in Q_k$, so
\[
Q_k u_i Q_k = Q_k u_{i^\sigma}Q_k.
\]
Thus \(\sigma\) fixes \(X\) setwise, and $x=(q_i)_{i\in Y_k}$ belongs to the base group.
Moreover, note that $\mathrm{supp}(x) \subseteq X$.

If $x \in {\bf Z}(P)$, since $P$ acts transitively on $Y_k$,
then all coordinates of $x$ corresponding to a direct factor in $Y_k$ are $1$, or none of them are.
This implies $x=1$.
On the other hand, if $x \notin{\bf Z}(P)$, then there exists $y\in P$ with $z=[x,y]\ne 1$.
As $P$ has class at most $2$, we have $z\in {\bf Z}(P)$.
As before, all coordinates of $z$ corresponding to a direct factor in $Y_k$ are $1$, or none of them are.
But $|\mathrm{supp}(z)| \leq 2 |\mathrm{supp}(x)| < |Y_k|$,
so $z=1$, which gives a contradiction.
\end{proof}

We can finally conclude the paper.

\begin{proof}[Proof of \cref{thrm:main}]
Let $\Gamma$ be a connected cubic graph and let $G \leqslant \Aut(\Gamma)$ act transitively on $V\Gamma$.
We work by induction on $|V\Gamma|$.
Let $\alpha \in \V\Gamma$, and
    apply Propositions~\ref{prop:arctran}, \ref{prop:minabel} and \ref{prop:minnonabelnontriv}.
   Finally, apply \cref{prop:minnonabel} with $P=G_\alpha$, noting that $(G_\alpha)^g = G_{\alpha^g}$ for any $g \in G$,
   so $G$ has base size $2$ in this case.
\end{proof}

\bibliographystyle{plain}
\bibliography{baseSize.bib}

\end{document}